\documentclass[10pt]{amsart}
\newtheorem{thm}{Theorem}[section]
\newtheorem{cor}[thm]{Corollary}
\newtheorem{lem}[thm]{Lemma}
\newtheorem{prop}[thm]{Proposition}

\theoremstyle{definition}
\newtheorem{defn}[thm]{Definition}
\theoremstyle{remark}
\newtheorem{rem}[thm]{Remark}
\numberwithin{equation}{section}
\newtheorem{ex}[thm]{\bf Example}
\newcommand{\norm}[1]{\left\Vert#1\right\Vert}

\def\star{^{*}}
\def\L2{L^{2}}
\def\undemi{\frac{1}{2}}
\def\M{\mathcal{M}}
\def\E{\mathcal{E}}
\def\B{\mathcal{B}}
\def\D{\mathcal{D}}
\def\N{\Bbb{N}}

\def\C{\Bbb{C}}

\def\e{\varepsilon}

\def\m1{^{-1}}

\def\C{\mathcal{C}}
\def\H{\mathcal{H}}
\def\F{\mathcal{F}}
\def\J{\mathcal{J}}
\def\K{\mathcal{K}}
\def\o{\omega}
\def\O{\Omega}
\def\P{\mathcal{P_+}}

\begin{document}

\title[]{Variations in noncommutative potential theory: finite energy states, potentials and multipliers.}%
\author{Fabio Cipriani, Jean-Luc Sauvageot}%
\address{Dipartimento di Matematica, Politecnico di Milano,
piazza Leonardo da Vinci 32, 20133 Milano, Italy.}%
\email{fabio.cipriani@polimi.it}%
\address{Institut de Math\'ematiques, CNRS-Universit\'e Pierre et Marie
Curie Universit\'e Paris VII, boite 191, 4 place Jussieu, F-75252 Paris Cedex 05}
\email{jlsauva@math.jussieu.fr}
\thanks{}
\subjclass{}%
\keywords{Dirichlet space, finite-energy functional, potential, carr\'{e} du champ, Dirichlet space multiplier.}%

\date{July 15, 2012}
\dedicatory{Dedicated to Gabriel Mokobodzki}

\begin{abstract}

In this work we undertake an extension of various aspects of the potential theory of Dirichlet forms from locally compact spaces to noncommutative C$^*$-algebras with trace. In particular we introduce finite-energy states, potentials and multipliers of Dirichlet spaces. We prove several results among which the celebrated Deny's embedding theorem and the Deny's inequality, the fact that the carr\'{e} du champ of bounded potentials are finite-energy functionals and the relative supply of multipliers.

\end{abstract}

\maketitle
\section{Introduction and description of the results.}

In the present work we develop further the potential theory of Dirichlet forms on noncommutative C$^*$-algebras with trace. We introduce and investigate {\it finite-energy states, potentials and multipliers}, objects naturally associated to Dirichlet spaces and which are meant to encode or reveal the geometric nature of the latter.
\vskip0.1truecm\noindent
In a companion work the results here obtained will be crucial to construct on C$^*$-algebras endowed with a Dirichlet form, the building blocks of a metric differential geometry (Dirac operators and Spectral Triples) and topological invariants (summable Fredholm modules in K-homology) in the framework of the Noncommutative Geometry developed by A. Connes \cite{Co}.
\vskip0.1truecm\noindent
Classical potential theory, studying harmonic functions on Euclidean spaces $\mathbb{R}^n$, finite-energy measures and their potentials, was based on the properties of kernel $|x-y|^{-1}$, the so called Green function, to understood as the integral kernel of the inverse of the Laplace operator (see \cite{Bre}, \cite{Ca}, \cite{Do}).
\par\noindent
In the late fifties, A. Beurling and J. Deny outlined, in two seminal papers \cite{BeDe1}, \cite{BeDe2}, the way to develop a kernel-free potential theory on locally compact Hausdorff spaces $X$. There, the central role was no more played by the Green function, but rather by quadratic forms which posses the fundamental {\it Markovian contraction property}
\begin{equation}
\E[a\wedge 1]\le \E[a]\, ,
\end{equation}
generalizing the Dirichlet integral of Euclidean spaces
\[
\E_{\mathbb{R}^n}[a]=\int_{\mathbb{R}^n}\, |\nabla a|^2\, dm\, .
\]
The second fundamental property these quadratic forms are required to have is {\it lower semicontinuity} on the algebra $C_0 (X)$.
Lower semicontinuity is a reminiscence of the fact that Dirichlet forms may represent energy functionals of physical systems (distributions of electric charges or quantum spinless particles in the ground state representation, for example). On the other hand this property allows, by a result of G. Mokobodzki \cite{Moko}, to extend the quadratic form to a lower semicontinuous form on the Hilbert spaces $L^2 (X,m)$, with respect to a wide family of Borel measures $m$ on $X$, giving rise to a positive self-adjoint generator $L$ of a Markovian semigroup $e^{-tL}$ on $L^2 (X,m)$
\[
\E[a]=\|L^{1/2} a\|^2_{L^2 (X,m)}
\]
Semigroups in this class are precisely the symmetric, strongly continuous, contractive, positivity preserving semigroups on $L^2 (X,m)$ which extend to weakly$^*$-continuous, contractive, positivity preserving semigroups on $L^\infty (X,m)$, symmetric with respect to the measure $m$.
\par\noindent
The $L^2$-theory is particularly interesting from at least two points of view. The first is that, as noticed by A. Beurling and J. Deny, there exists a one to one correspondence between Dirichlet forms and Markovian semigroups on $L^2 (X,m)$. The second is that these objects are also in one to one correspondence with Hunt's Markov stochastic processes $(\mathbb{E}_x , \omega_t )$ on $X$, which are symmetric with respect to $m$
\[
(e^{-tL} f)(x)=\mathbb{E}_x (f(\omega_t))\qquad x\in X\, ,\quad t\in [0,+\infty )\, .
\]
The third requirement a Dirichlet form $\E$ on $L^2 (X,m)$ has to satisfies is called {\it regularity}, and concerns the existence of a form core which is also a dense sub-algebra of $C_0(X)$. This allows to develop a rich theory of finite-energy measure and their potentials and, in particular, the construction of a Choquet capacity on the space $X$. Sets having vanishing capacity can be considered to be negligible from the point of view of Potential Theory and M. Fukushima made a crucial use of them to construct the essentially unique Hunt's process on $X$ associated to the regular Dirichlet form (see \cite{F1}, \cite{F2}, \cite{FOT}).
\vskip0.1truecm\noindent
The idea to generalize the notion of Markovian semigroup to C$^*$-algebras $A$ more general than the commutative ones, which are necessarily of type $C_0 (X)$, arose in Quantum Field Theory when L. Gross [G1], [G2] approached the problem of the existence and uniqueness of the ground state of an assembly of $\frac{1}{2}$spin particles, in terms of certain hypercontractivity properties of the Markovian semigroup on the Clifford C$^*$-algebra of an infinite dimensional (one-particle) Hilbert space, generated by the Hamiltonian operator.
\par\noindent
Later, S. Albeverio and R. Hoegh-Krhon [AHK1] introduced Dirichlet forms on C$^*$-algebras with trace $(A,\tau)$ as closed, quadratic forms on the G.N.S. Hilbert space $L^2 (A,\tau)$, satisfying a suitable contraction property generalizing (1.1) and having a form core which is a dense sub-algebra of $A$. They also generalized the Beurling-Deny correspondence between Dirichlet forms and Markovian semigroups on $L^2 (A,\tau)$. This theory was subsequently developed by J.-L. Sauvageot [S2], E.B. Davies and M. Lindsay [DL]. Applications were found in Riemannian Geometry by E. B. Davies and O. Rothaus [DR1,2] to spectral bounds for the Bochner Laplacian and in Noncommutative Geometry by J.-L. Sauvageot [S3,4] to the transverse heat semigroup on the C$^*$-algebra of a Riemannian foliation.
\vskip0.1truecm\noindent
The discovery of the differential calculus underlying the structure of Dirichlet forms [S2], [CS1], allows to represent them as
\[
\E[a]=\|\partial a\|_\H^2
\]
in terms of an essentially unique derivation $\partial$ on $A$ taking its values in a Hilbert $A$-bimodule $\H$. The derivation thus appears as a differential square root of the generator
\[
L=\partial^*\circ\partial\, .
\]
This differential calculus allowed a potential theoretic characterization of Riemannian manifolds having a positive curvature operator as those for which the semigroup generated by the Dirac Laplacian on the Clifford C$^*$-algebra is Markovian \cite{CS3}.
\vskip0.1truecm\noindent
Among the others applications of Dirichlet forms and their differential calculus on a C$^*$-algebra with trace, we mention the use made by D. Voiculescu \cite{V1}, \cite{V2} and Ph. Biane [Bi] in Free Probability Theory to define and investigate Free Entropy and the recent appearance in K-theory of Banach algebras \cite{V3} and in K-homology of fractals \cite{CGIS1}, \cite{CGIS2}.
\par\noindent
Derivations and their associated Markovian semigroups and resolvent has been used by J. Peterson to approach $L^2$-rigidity in von Neumann algebras \cite{Pe1}, \cite{Pe2} to characterize  von Neumann algebras having the property T (a generalization of the Kazhdan property T for groups) and by Y. Dabrowski to prove the property non-$\Gamma$ of von Neumann algebras generated by noncommuting self-adjoint generators under finite nonmicrostates free Fisher information, still in the framework of D. Voiculescu Free Entropy theory \cite{Da}. Markov semigroups and Dirichlet forms appear in connection with L\'{e}vy's processes on Compact Quantum Groups \cite{CFK}.
\vskip0.1truecm\noindent
The paper is organized as follows. In Section 2 we recall the basic definitions and properties of Dirichlet forms $\E$, their Dirichlet spaces $\F$, Markovian semigroups and resolvents on C$^*$-algebras with traces. In Section 3 we introduce finite-energy functionals and potentials associated to Dirichlet spaces. We prove a correspondence between these classes of objects, the positivity of potentials and a version of a ''noncommutative maximum principle". As an important tool, we introduce the fine C$^*$-algebra $\C$, intermediate among the C$^*$-algebra $A$ and the von Neumann algebra $\M$, to which finite-energy functionals automatically extend. The section contains also a detailed discussions of a class of examples on the reduced C$^*$-algebra $C^*_{red} (G)$ of a locally compact group associated to negative definite functions on them. In Section 4 we provide a version, in our noncommutative framework, of a Deny's embedding theorem by which the Dirichlet space $\F$ can be continuously embedded in the G.N.S. space $L^2 (A,\omega)$ of any finite-energy state $\omega$ whose potential is bounded. We prove also a version of the Deny's inequality. In Section 5, making use of the canonical differential calculus associated to Dirichlet spaces, we recall the definition of energy functionals or carr\'{e} du champ $\{\Gamma[a]\in A^*_+ :a\in\F\}$ associated to a Dirichlet space and we show that the energy functional $\Gamma [G]$ of bounded potential $G\in\P$ is a finite-energy functional. In the last Section 6, we introduce {\it multipliers} of a Dirichlet space and show that bounded potentials $g\in\P$ whose energy functional $\Gamma[g]$ has a bounded potential $G(\Gamma[g])\in\P$ is a multiplier. This show a relative abundance of multipliers and, in particular, that bounded potentials can be approximated by potentials that are also multipliers.
\vskip0.1truecm\noindent
The content of this work has been the subject of the following talks: Workshop ''Noncommutative Potential Theory'' Besan\c{c}on January 2011, GDRE-GREFI-GENCO Meeting Institut H. Poincar\'{e} Paris June 2012, INDAM Meeting ''Noncommutative Geometry, Index Theory and Applications'' Cortona-Italy, June 11-15 2012.

\section{Dirichlet forms on C$^*$-algebras}

In this section we summarize the main definitions and some fundamental results of the theory of noncommutative Dirichlet forms on C$^*$-algebras with trace, for which one may refer to [AHK], [C2], [CS1], [DL].
\subsection{C$^*$-algebras, traces and their standard forms}
Let us denote by $(A,\tau)$ a separable C$^*$-algebra $A$ and a densely defined, faithful, semifinite, lower semicontinuous, positive trace on it.
\par\noindent
 We denote by $L^2 (A, \tau )$ the Hilbert space of the
Gelfand--Naimark--Segal (G.N.S.) representation $\pi_\tau$ associated to $\tau$, and by $\M$ or $L^\infty (A, \tau )$ the von Neumann algebra ${\pi_\tau (A)}^{\prime\prime}$ in $\mathbb{B} (L^2 (A, \tau ))$ generated by $A$ through the G.N.S. representation.
\par\noindent
When unnecessary, we shall not distinguish between $\tau$ and its canonical normal extension on $\M$, between elements of $A$ and their representation in $\M$ as a bounded operator in $L^2 (A, \tau )$, nor between elements $a$ of $A$ or $\M$ which are square integrable, in the sense that $\tau(a^* a)<+\infty$, and their canonical image in
$L^2(A, \tau )$.
\par\noindent
Then $\norm a$ stands for the uniform norm of $a$ in $A$ or in $\M$, $\norm \xi _2$ or $\norm \xi _{L^2(A, \tau )}$ for the norm of $\xi\in L^2(A, \tau )\,$ and $1_\M$ for the unit of $\M$.
\par\noindent
As usual $A_+$, $\M_+$ or $L^\infty_+ (A, \tau )$ and $L^2_+ (A, \tau )$ will denote the positive part of $A$, $\M$ and $L^2(A, \tau )$ respectively.
\par\noindent
Recall that ($\M, L^2(A, \tau ), L^2_+(A, \tau )$) is a {\it standard form} of the von Neumann algebra $\M$ (see \cite{Ara}). In particular $L^2_+(A, \tau )$ is a self-polar, closed convex cone in $L^2(A, \tau )$, inducing an anti-linear isometry (the modular conjugation) $J$ on $L^2(A, \tau )$ which is an extension of the involution $a\mapsto a^*$ of $\M$. The subspace of $J$-invariant elements (called {\it real}) will be denoted by $L^2_h(A, \tau )$ (cf. [Dix]). Any element $\xi\in L^2 (A,\tau)$ can written uniquely as $\xi=\xi_r+i\xi_i$ for real elements $\xi_r,\xi_i\in L^2_h (A,\tau)$ and any real element $\xi\in L^2_h (A,\tau)$ can written uniquely as $\xi=\xi_+ - \xi_-$ for orthogonal positive elements $\xi_\pm\in L^2_+ (A,\tau)$, called the positive and negative parts. Recall that $\xi_+$ is the Hilbert projection of $\xi\in L^2_h (A,\tau)$ onto the closed convex set $L^2_+(A, \tau )$. For a real element $\xi\in L^2_h (A,\tau)$, the positive element $|\xi|:=\xi_+ +\xi_-\in L^2_+ (A,\tau)$ will be called the modulus of $\xi$.
\par\noindent
Whenever $\xi\in L^2_h (A, \tau )$ is real, the symbol $\xi\wedge 1$ will denote its Hilbert projection onto the closed and convex subset  $C$ of $L^2_h (A, \tau )$ obtained as the $L^2$--closure of $\{a\in A\cap L^2 (A,\tau): a\le 1_\M\}$.

\subsection{${\rm C}^*$--Dirichlet forms, Dirichlet spaces and Dirichlet algebras}

Let $\mathbb{M}_n (\mathbb{C})$ be, for $n\ge 1$, the C$^*$--algebra of $n\times n$ matrices with complex entries, $1_n$ its unit, I$_n$ its identity automorphism and $tr_n$ its normalized trace. For every $n\ge 1$, we will indicate by $\tau_n$ the trace $\tau\otimes tr_n$ of the C$^*$--algebra $\mathbb{M}_n (A)=A\otimes \mathbb{M}_n (\mathbb{C})$ of $n\times n$ with entries in $A$.
\vskip0.2truecm\noindent
The main object of our investigation is the class of C$^*$--{\it Dirichlet forms} on $L^2 (A, \tau )$ whose definition we recall here (cf. [AHK], [DL], [C1], [CS1]).
\begin{defn}[C$^*$-Dirichlet forms]
A closed, densely defined, nonnegative quadratic form $(\E,\F)$ on $L^2(A, \tau )$ is said to be:
\vskip0.1truecm\noindent
\item{i)} {\it real} if
\begin{equation}
J(\xi)\in \F,\quad \E [J(\xi)]=\E [\xi]\qquad  \xi\in \F\, ,
\end{equation}
\item{ii)} a {\it Dirichlet form} if it is real and {\it Markovian} in the sense that
\begin{equation}
\xi\wedge 1\in \F,\quad \E [\xi\wedge 1 ]\le\E [\xi]\qquad  \xi\in\F\cap L^2_h(A, \tau )\, ,
\end{equation}
\item{iii)} a {\it completely Dirichlet form} if the canonical extension
$(\E_n ,\F_n ))$ to $L^2 (\mathbb{M}_n (A), \tau_n )$
\begin{equation}
\E_n [[\xi_{i,j}]_{i,j=1}^n)] :=\sum_{i,j=1}^n \E[\xi_{i,j}]\qquad
[\xi_{i,j}]_{i,j=1}^n\in \F_n:= \mathbb{M}_n (\F)\, ,
\end{equation}
is a Dirichlet form for all $n\ge 1$\, ,
\item{iv)} a C$^*$-{\it Dirichlet form} if it is a completely Dirichlet form which is {\it regular} in the sense that the subspace $\B := A\cap \F$ is dense in the C$^*$--algebra A and is a {\it form core} for $(\E ,\F)$.
\end{defn}\noindent
Notice that, in general, the property
\[
|\xi|\in \F,\quad \E [\,|\xi|\, ]\le\E [\xi]\qquad  \xi\in \F\cap L^2_h(A, \tau )
\]
is a consequence of the property (2.2) and that it is actually equivalent to it
when $\tau$ is finite, the cyclic and separating vector $\xi_\tau$ representing $\tau$ belongs to $\F$ and $\E [\xi_\tau]=0$ (see [C1]).
\par\noindent
\begin{rem}
Even if in this paper we formulate the results in the setting of the G.N.S. standard form of $(A,\tau)$, they can be equivalently stated and proved in a general standard form of $(A,\tau)$ (see [C1]). This may be an important advantage when considering specific examples where an {\it ad hoc} standard form can be more manageable that the G.N.S. one.
\end{rem}

\vskip0.2truecm
\centerline{{\it To simplify notations, in the rest of the paper}}
\centerline{{\it"Dirichlet form" will always mean C$^*$-Dirichlet form.}}
\vskip0.2truecm

We will denote by $(L, D(L)$ the densely defined, self-adjoint, nonnegative operator on $L^2(A,\tau)$ associated with the closed quadratic form $(\E ,\F)$
\begin{equation}
\E[\xi]=||L^{1/2}\xi||^2\qquad \xi\in \F=\D(L^{1/2})\, .
\end{equation}
This operator is the generator of the strongly continuous, contractive semigroup $\{e^{-tL} :t\ge 0\}$ on the Hilbert space $L^2 (A,\tau)$. This semigroup is Markovian in the sense that it is positivity preserving and extends to a weakly$^*$-continuous semigroup of contractions on the von Neumann algebra $\M$. By duality and interpolation this semigroup extends also as a strongly continuous, positivity preserving, contractive semigroup on the noncommutative $L^p$-space $L^p (A,\tau)$ for each $p\in [1,+\infty]$.
\vskip0.2truecm\noindent
As practice, several aspects of potential theory are more easily managed working with the resolvent family $\{(I+\e L)^{-1} :\e\ge 0\}$ than using the semigroup itself. in particular, we will make use of the following obvious properties.
\begin{lem} 
For $\e>0$, the resolvent $(I+\e L)^{-1}$ is a symmetric contraction in $L^2(A,\tau)$ which operates as a $\sigma$-weakly continuous, completely positive, contraction of the von Neumann algebra $\M$ and converges strongly to the identity on $\F$.
\end{lem}

\begin{defn}[Dirichlet spaces, Dirichlet algebras and their fine C$^*$-algebras]
\par\noindent
The domain $\F$ of the Dirichlet form will called {\it Dirichlet space} when considered as a Hilbert space endowed with its {\it graph norm}
\begin{equation}
||\xi||_\F :=\bigl(\E[\xi]+||\xi||^2_{L^2(A,\tau)}\bigr)^{1/2}\qquad \xi\in\F
\end{equation}
and the scalar product
\begin{equation}
\langle\xi ,\eta\rangle_\F :=\E (\xi ,\eta)+ (\xi ,\eta)_{L^2(A,\tau)}\qquad \xi ,\eta\in\F\, .
\end{equation}
The subspace $\B :=\F\cap A$ is in fact an involutive, sub-algebra of $A$ called the {\it Dirichlet algebra} (see [DL], [C2]). By the regularity assumption, it is dense in the Dirichlet space $\F$ as well in the C$^*$-algebra $A$, with respect to their own topologies.
\vskip0.2truecm\noindent
The subspace $\widetilde \B :=\F\cap \M$ is an involutive sub-algebra of $\M$ called the {\it extended Dirichlet algebra}. It is dense in the Dirichlet space $\F$ as well in the von Neumann algebra $\M$ with respect to its $\sigma$-weak topology.
\vskip0.2truecm\noindent
In our approach to potential theory on noncommutative C$^*$algebras, a distinguished role will be played by the {\it fine C$^*$-algebra} $\mathcal{C}\supseteq A$, closure of the extended Dirichlet algebra $\widetilde \B$ in the norm topology of the von Neumann algebra $\M$. In particular, we will make use of the fact that the Dirichlet form $(\E ,\F)$, originally assumed to be regular with respect to the C$^*$-algebra $A$, is still regular with respect to the larger fine C$^*$-algebra $\mathcal{C}$ (see Section 5 below).
\end{defn}

We conclude this section with three examples of Dirichlet space. In the first we recall the classical Beurling-Deny theory on locally compact spaces $X$, where the C$^*$-algebra $A$ is the commutative algebra $C_0 (X)$ of continuous functions vanishing at infinity endowed with its uniform norm. The second one deals with typical situations in harmonic analysis where the (reduced) group C$^*$-algebra C$^*_{red} (G)$ of a locally compact group $G$ is most of the time noncommutative. The third illustrates the standard Dirichlet form on noncommutative tori.

\begin{ex}[Dirichlet spaces on commutative C$^*$-algebras]
By a fundamental result of I.M. Gelfand (see [Dix]), commutative C$^*$-algebras are of type $C_0 (X)$ for a suitable locally compact, Hausdorff space $X$. In this case, positive maps are automatically completely positive so that positive or Markovian semigroup are automatically completely positive or Markovian and all Dirichlet forms are automatically completely Dirichlet forms. In the commutative case our framework thus coincides with that introduced by A. Beurling and J. Deny [BeDe2] to develop potential theories on locally compact Hausdorff spaces.
\par\noindent
The model Dirichlet form on the Euclidean space $\mathbb{R}^n$ or, more generally, on any Riemmannian manifold $M$ endowed with its Riemannian measure $m$, is the Dirichlet integral
\[
\E[a]=\int_M |\nabla a|^2\, dm\qquad a\in L^2 (M,m)\, .
\]
In this case the trace on $C_0 (M)$ is given by the integral with respect to the measure $m$ and the Dirichlet space is the Sobolev space $H^1 (M)\subset L^2 (M,m)$.
\par\noindent
\par\noindent
Much of the potential theory of Dirichlet forms on locally compact spaces, relies on a notion of smallness for subsets of $X$ called polarity. This can be expresses in terms of a set function called capacity (see \cite{FOT}). In the present noncommutative setting, it will be the fine C$^*$-algebra $C\subseteq\M$ to play the role of the Choquet capacity (see Lemma \ref{finecalculus} below).
\end{ex}

\begin{ex}[Dirichlet spaces on group C$^*$-algebras]

Let $G$ be a locally compact, unimodular group, with unit $e\in G$, whose elements will be denoted by $s,t,\dots$, and let $ds$ be a Haar measure on it.
Denote by $\lambda_G$ its left regular representation on $L^2 (G, ds)$ acting by
\[
(\lambda_G (s)a)(t):=a(s^{-1}t)\qquad s,t\in G\, ,\quad a\in L^2 (G,ds)
\]
and by $C\star_{red}(G)$ its reduced C$\star$-algebra in $\mathbb{B}(L^2(G,ds))$ generated by $\{\lambda_G (s)\in\mathbb{B}(L^2 (G,ds)) : s\in G\}$ (see [Dix]).
More explicitly, for $a,b\in C_c (G)\subseteq C^*_{red}(G)$ their product is defined by convolution
\[
(a\ast b)(s):=\int_G a(t)b(st^{-1})\, dt\quad s\in G
\]
while involution is defined by
\[
(a^*)(s):=\overline{a(s^{-1})}\qquad s\in G\, .
\]
The left regular representation of $G$ extends to a $^*$-representation of the reduced C$^*$-algebra and will be denoted by the same symbol. The functional $C\star_{red}(G)\supseteq C_c (G)\ni a\mapsto a(e)\in\mathbb{C}$ extends to a trace state on $C\star_{red}(G)$ and the associated G.N.S. representation coincides with the left regular representation above. In particular the G.N.S. Hilbert space $L^2 (C\star_{red}(G),\tau )$ can be identified with $L^2 (G,ds)$ and its positive cone with the cone of positive definite, square integrable functions.
\par\noindent
Any positive, conditionally negative definite function $\ell :G\to [0,+\infty)$ (see for example \cite{CCJJV}) gives rise to a regular Dirichlet form
\[
\E_\ell [a] = \int_G |a(s)|^2 \ell(s)\, ds\, ,
\]
with domain the space of those $a$ in $L^2 (G,ds)$ for which the integral converges (see [CS1], [C2]).
\vskip0.2truecm\noindent
Examples of the above framework arise on $\mathbb{Z}^n$, where as negative definite function one can choose the Euclidean length $\ell (k):=|k|$ or its square $\ell (k):=|k|^2$, and on free groups $\mathbb{F}_n$ with $n\in\{1,2,\dots\}$ generators where the most important negative definite functions are the length functions associated to systems of generators (see \cite{Haa1}).

%
\end{ex}

\begin{ex}{\bf Dirichlet forms on noncommutative tori.}
Noncommutative tori are a family of C$^*$-algebras which represent a sort of gymnasium for Noncommutative Geometry [Co].
They are defined, for any fixed irrational $\theta\in [0,1]$, as the universal C$^*$-algebras $A_\theta$ generated by two unitaries $U$ and $V$,
satisfying the relation
\[
VU=e^{2i\pi\theta}UV\, .
\]
The functional $\tau : A_\theta \to\mathbb{C}$ given by
\[
\tau(U^n V^m)=\delta_{n,0}\delta_{m,0} \qquad n,m\in \mathbb{Z}
\]
is a tracial state and the {\it heat semigroup} $\{T_t :t\ge 0\}$ on $A_\theta$ is defined by
\[
T_t(U^nV^m)= e^{-t(n^2+m^2)}U^nV^m \qquad n,m\in
\mathbb{Z}\, .
\]
It is $\tau$-symmetric and the associated Dirichlet form is the closure of the quadratic form given by
\[
\E\Bigl[\sum_{n,m\in\mathbb{Z}}\alpha_{n,m}U^nV^m\Bigr] =
\sum_{n,m\in\mathbb{Z}}(n^2+m^2)|\alpha_{n,m}|^2
\]
defined on the algebra $\{\sum_{n,m\in\mathbb{Z}}\alpha_{n,m}U^nV^m\in A_\theta :[\alpha_{n,m}]_{n,m\in\mathbb{Z}}\in c_c (\mathbb{Z}^2)\}$
\end{ex}

\section{Finite-energy functionals and potentials.}

In this section we introduce two of the main objects of our investigation: the class of finite-energy functionals and the class of potentials of a Dirichlet space. These are generalizations to possibly noncommutative C$^*$-algebras of the corresponding objects introduced by A.Beurling and J. Deny in their work on Dirichlet forms on locally compact spaces [BeDe2].

\begin{defn}[Finite-energy functionals and potentials]

Let $(\E ,\F)$ be a Dirichlet form on the separable C$^*$-algebra $(A,\tau)$ endowed with a densely defined, faithful, semifinite, lower semicontinuous, positive trace.

\begin{itemize}
\item A positive functional $\omega\in A^*_+$ will be said to be a {\it finite-energy functional} if
\begin{equation}
\omega (b)\leq c_\omega \|b\|_\F\qquad b\in\B_+
\end{equation}
for some $c_\omega \ge 0$.
\item An element $\xi\in\F$ will be called a {\it potential} if
\begin{equation}
\langle\xi ,b\rangle_\F\ge 0\qquad b\in\B_+ :=\B\cap L^2_+ (A,\tau)\, .
\end{equation}
\item Let $\omega\in A^*_+$ be a finite-energy functional. By the regularity of the Dirichlet form, in particular the fact that the Dirichlet algebra $\B$ is a form core, the exists a unique element $\xi\in\F$ determined by the the relation
\begin{equation}
\omega (b)=\langle \xi,b\rangle_\F =\E (\xi,b)+(\xi,b)_2\qquad b\in\B\, .
\end{equation}
The element $\xi$ will be called the {\it potential associated with $\omega$} and will be denoted by $G(\omega)$.
\end{itemize}
Thus, finite-energy functionals and their potentials satisfy the relation
\begin{equation}
\omega (b)=\langle G(\omega) ,b\rangle_\F\qquad b\in\B\, .
\end{equation}
Moreover, by the formula above, any finite-energy functional can then be extended to the whole Dirichlet space $\F$, the quantity
\begin{equation}
\E [\omega]:=\E [G(\omega)]=\omega (G(\omega))
\end{equation}
is called the {\it energy content} of $\omega$ and one has $|\omega (b)|\le \sqrt{\E[\omega]} \|b\|_\F$ for all $b\in \F$.
\end{defn}
\par\noindent
The set $\P$ of potentials is, by definition, the polar cone of the positive cone $\F_+ :=\F\cap L^2_+ (A,\tau)$ in the Dirichlet space,
\[
\P := \F_+^\circ =\{\xi\in \F : \langle\xi , \eta\rangle_\F \ge 0\, \,\, {\rm for\,\,all\,\, \eta\in\F_+}\}\, .
\]
We will prove in Proposition \ref{positivity} below that potentials are necessarily positive elements of $L^2_+ (A,\tau)$ so that $\P\subseteq \F_+$ and then $\P\subseteq \P^\circ$.

\begin{ex}[Finite-energy normal functionals] Let $h\in L^2_+ (A,\tau)\cap L^1 (A,\tau)$ and consider the normal positive functional $\omega_h\in\mathcal{M}_{*+}$ defined by
\[
\omega_h (b):=\tau (hb)\qquad b\in \mathcal{M}\, .
\]
Since $h\in L^2 (A,\tau)$ then $\xi:=(I+L)^{-1}h\in\F$ is such that
\[
\big<\xi\,,\,b\big>_\F=(L^{1/2}\xi\,,\,L^{1/2}b)+(\xi\,,\,b) = \tau(hb) \qquad b\in \B\, ,
\]
the vector $\xi\in\F$ is a potential, the normal positive linear form $\omega_h$ is a finite-energy functional, its potential coincides with $\xi$
\[
G(\omega_h)=(I+L)^{-1}h
\]
and its energy content is given by $\E[\omega_h]=\omega_h ((I+L)^{-1}h)=\tau (h(I+L)^{-1}h)$.
\end{ex}


\begin{ex}[Finite-energy functionals and potentials on group C$^*$-algebras]

Let us consider the Dirichlet form on a group algebra C$^*_{red}(\Gamma)$ of a discrete group $\Gamma$ associated to negative definite function $\ell :\Gamma\to [0,+\infty)$, as in Example 2.5,
\[
\E_\ell [a]=\sum_{s\in\Gamma} \ell (s)|a(s)|^2\qquad a\in l^2 (\Gamma)\, .
\]
In this case $\o$ is a finite-energy state on C$^*_{red}(\Gamma)$ if and only if
\[
\sum_{s\in\Gamma} \frac{|\varphi_\omega (s)|^2}{1+\ell (s)} < +\infty
\]
and its potential $G(\omega)$ is given by
\[
G(\omega)(s) = \frac{\varphi_\omega (s)}{(1+\ell (s))}\qquad s\in \Gamma\, ,
\]
where $\varphi_\omega :\Gamma\to\mathbb{C}$ is the normalized, positive definite function associated to the state $\omega$ and defined as $\varphi_\omega (s):=\omega (\delta_s)$ for all $s\in\Gamma$.
In particular the energy content of $\omega$ is equal to
\[
\E_\ell [\omega ]= \E_\ell [G(\omega)] = \sum_{s\in\Gamma} \frac{|\varphi_\omega (s)|^2}{1+\ell (s)}\, .
\]
In other words, since states $\omega$ on $C\star_{red}(\Gamma)$ are characterized by the fact that the associated function $\varphi_\omega$ is positive definite (see [Dix]), potentials $\xi\in \P$ associated to the Dirichlet form $\E_\ell$ have the form
\[
\xi (s)=\frac{\varphi_\xi (s)}{1+\ell (s)}\qquad s\in\Gamma
\]
for some positive definite function $\varphi_\xi :\Gamma\to \mathbb{C}$. Notice that, since $\ell$ is a negative definite function, the function $(1+\ell)^{-1}$ is positive definite so that the potential $\xi$ is a positive definite element of $L^2 (G)$. It will be shown later in this section that positivity of potentials is a general fact valid in all Dirichlet spaces.
\par\noindent
On groups having the Kazhdan property T, all negative definite function are bounded so that the cone of potential associated to any such negative definite function $\ell$ simply coincides with the cone of square integrable, positive definite functions. Richer classes of examples can be found on groups having the Haagerup property, where there exist proper, negative definite functions (see for example \cite{CCJV}).
\vskip0.2truecm\noindent
Suppose that $\Gamma$ has polynomial growth (i.e. by a theorem of M. Gromov, it has a nilpotent subgroup of finite index) so that, with respect a system of generators $S\subset \Gamma$, the associated length function $\ell_S$, assumed to be negative definite, has spherical growth $\sigma_S :\mathbb{N}\to\mathbb{N}$ behaving as $\sigma_S (k)\sim k^{d-1}$ for some $d>1$. If $\Gamma$ is nilpotent, by a theorem of J. Dixmier, the exponent $d$ coincides with the homogeneous dimension $d(\Gamma)$, defined in terms of the relative indexes of its lower central series (see \cite{CCJJV}). Then
\[
\|(1+\ell)^{-1}\|^q_{\ell^q (\Gamma)} = \sum_{s\in\Gamma}(1+\ell (s))^{-q} = \sum_{k\in\mathbb{N}} (1+k)^{-q} \sigma_S (k)<+\infty
\]
for all $q>d$. If $\omega\in A^*_+$ is a (pure) state whose cyclic (irreducible) representation is $l^p (\Gamma)$-integrable for some $2\le p<\frac{2d}{d-1}$, by definition this means that $\varphi\in l^p (\Gamma)$, then, by the H\"older inequality, it is a finite-energy state with respect to the Dirichlet form $\E_l$
\[
\E_\ell [\omega ]= \E_\ell [G(\omega)] = \sum_{s\in\Gamma} \frac{|\varphi_\omega (s)|^2}{1+\ell (s)}\le
\|\varphi_\omega\|_{\ell^p (\Gamma)}\cdot\|(1+\ell)^{-1}\|^q_{\ell^q (\Gamma)} <+\infty\, .
\]
For a specific example one may consider the Heisenberg group which is nilpotent with homogeneous dimension $d(\Gamma)=4$.
\vskip0.2truecm\noindent
As $\ell$ is a negative definite function, so is its square root $\sqrt\ell$. Hence $(1+\sqrt\ell)^{-1}$ is a positive definite, normalized function and there exists a state $\omega_\ell\in A^*_+$ such that $\varphi_{\omega_\ell} (s)=(1+\sqrt\ell (s))^{-1}$ for all $s\in\Gamma$. Since
\[
(1+\sqrt x)^2 \le 2(1+x)\le 2(1+\sqrt x)^2\qquad x>0\, ,
\]
a functional $\omega\in A^*_+$ is a finite-energy state if and only if
\[
\sum_{s\in\Gamma} \frac{|\varphi_\omega (s)|^2}{(1+\sqrt\ell (s))^2}=\sum_{s\in\Gamma} |\varphi_{\omega_\ell} (s)\cdot\varphi_\omega (s)|^2 <+\infty\, .
\]
Notice that $\varphi_{\omega_\ell}\cdot\varphi_{\omega}$ is a coefficient of a cyclic sub-representation of the tensor product $\pi_{\omega_\ell}\otimes\pi_\omega$ of the cyclic representations $(\pi_\ell ,\H_\ell ,\xi_\ell)$ and $(\pi_\omega ,\H_\omega ,\xi_\omega)$ associated to the states $\omega_\ell$ and $\omega$.
Hence if $\omega$ is a finite-energy state, the representation $\pi_{\omega_\ell}\otimes\pi_\omega$ is not disjoint from the left regular representation $\lambda_\Gamma$.
\par\noindent
Moreover, since a state $\omega$ has finite energy with respect to the Dirichlet form generated by a negative definite function $\ell$ if and only if it is a finite energy state with respect to the Dirichlet forms associated to each negative type functions $\lambda^{-2}\ell$ for all $\lambda >0$, we have that the family of normalized, positive definite functions $\{\varphi_\lambda :=\varphi_{\omega_{\lambda^{-2}\ell}}\cdot\varphi_\omega:\lambda >0\}$, explicitly given by
\[
\varphi_\lambda (s)=\frac{\lambda}{\lambda +\sqrt{\ell (s)}}\cdot\varphi_\omega (s)\qquad s\in\Gamma\, ,
\]
generates a family of cyclic representations $\{\pi_\lambda :\lambda >0\}$, contained in the left regular representation $\lambda_\Gamma$ which interpolate between the left regular representation $\lambda_\Gamma$ and the cyclic representation $\pi_\omega$ associated to the finite energy state $\omega$. In fact
\[
\lim_{\lambda\to 0^+} \varphi_\lambda = \delta_e\, ,\qquad \lim_{\lambda\to +\infty} \varphi_\lambda = \varphi_\omega
\]
pointwise.
\end{ex}

Now we prove that finite-energy functionals extends to positive functionals on the fine C$^*$-algebras $\mathcal{C}$. For this we need the following approximation result.

\begin{lem} \label{BdansBtilde}
Let $b\in \widetilde \B$ such that $b^*=b$. Then there exists a sequence of self-adjoint elements $\{b_n\}_{n\in \N}\subset\B$ such that $||b_n-b||_\F\to 0$, $||b_n||\leq ||b||$ and $b_n\to b$ $\sigma$-weakly in $\M$. If $\beta\geq 0$, one can get $b_n\geq 0$ for all $n$.
\end{lem}

\begin{proof} As, by the regularity of $(\E ,\F)$, the Dirichlet algebra $\B$ is a form core, there exists a sequence $\{b_n\}_{n\in \N}\subset\B$ which converges to $b$ in $\F$. By reality (2.1) of $\E$, the sequence $b_n^*$ converges also to $b^*$, so that one can suppose $b_n=b_n^*$ for all $n$.
\par\noindent
Set $K:=||b||$ and, for each $n$, let $e_n$ be the spectral projection of $b_n$ corresponding to the interval $(-\infty,K]$. Set $b'_n=b_n \wedge K=e_n\beta_n+K(I-e_n)$. One has $||b'_n||_{L^2(A,\tau)}\leq ||b_n||_{L^2(A,\tau)}$ (since ${b'_n}^2\leq b_n^2$) and, by the Markovian property (2.2) of the Dirichlet form, $\E[b'_n]\leq \E[b_n]$. Hence, the sequence $b'_n$ is bounded in $\F$. Replacing it by a subsequence, one can suppose that it has a weak limit $\gamma$ in $\F$, with $\gamma\leq b$.
\par\noindent
As $b'_n\to \gamma$ weakly in $L^2(A,\tau)$, we have
\begin{equation}\label{ineqgamma}
\tau(\gamma^2)\leq \liminf \tau({b'_n}^2)\leq \lim \tau(b_n^2)=\tau(b^2)
\end{equation}
and, weakly in $L^2(A,\tau)$,
\begin{equation}\label{betamoinsgamma}
(\beta_n-KI)(I-e_n)=b_n-b_n\wedge K\to b-\gamma\,.
\end{equation}
As $K^2\tau(I-e_n)\leq \tau(b_n^2(I-e_n))\leq \tau(b_n^2)\to \tau(b^2)$, one can suppose that the $I-e_n$ have a weak limit $p$ in $L^2(A,\tau)$, which is also a $\sigma$-weak limit in $\M$. So, $b_n (1-e_n)$ converges weakly to $b p$ in $L^2(A,\tau)$ and (\ref{betamoinsgamma}) provides
$$(b-KI)p=b-\gamma\,.$$
As $b_n$ commute with $e_n$, $b$ will commute with $p$, so that, in this equality, the left hand side is a negative operator while the right hand side is a positive operator. This proves $\gamma=b$ and, by (\ref{ineqgamma}), that $b'_n\to b$ strongly in $L^2(A,\tau)$. As the sequence $b'_n$ is bounded in $\F$, it converges to $b$ weakly in $\F$. As moreover $\E[b'_n]\leq \E[b_n]$ which converges to $\E[b]$, this must be a strong limit in $\F$.

\medskip\noindent Similarly, $b_n=b'_n\vee (-K)=-(-b_n\wedge K)$ converges to $b$ in $\F$. It is a bounded sequence in $\M$, with norm less that $K=\|b\|$. As its only possible $\sigma$-weak limit is $b$, it converges to $b$ $\sigma$-weakly in $\M$.

\smallskip\noindent Note that, if $b\geq 0$, one can replace $b_n=b'_n\vee (-K)$ by $b_n=b_n \vee 0$, so that $b_n\geq 0$ for all $n$.

\end{proof}

\begin{prop}\label{omegatilde}
If $\omega\in A^*_+$ is a finite-energy functional, then the linear map $\widetilde \omega :\widetilde \B \rightarrow\mathbb{C}$
\begin{equation}
\widetilde \omega (b):=\big<G(\omega)\,,\,b\big>_\F
\end{equation}
extends to the C$^*$-algebra $\mathcal C$ as a positive map with norm equal to $\|\omega\|_{A^*}$.
\end{prop}

\begin{proof}
Note first that $G(\omega)^*=G(\omega)$ since, by symmetry of $\E$, one has, for $b\in \B$:
\[
\big<G(\omega)^*,b\big>_\F= \big<b^*,G(\omega)\big>_\F=\overline{\omega(b^*)}= \omega(b)=\big<G(\omega),b\big>_\F\,.
\]
The same computation proves that $\widetilde \omega$ is hermitian: $\widetilde \omega(b^*)=\overline {\widetilde \omega(b)}$ for $b\in \widetilde \B$.
\par\noindent
Let $b=b^*\in \widetilde \B$ and $b_n$ a sequence in $\B$ provided by Lemma \ref{BdansBtilde}. Since any finite energy functional is continuous with respect to the topology of $\F$, one has
\[
|\widetilde \omega(b)|=\lim| \omega(b_n)|\leq \|\omega\|_{A^*} \limsup \|b_n\|_A \leq \|\omega\|_{A^*}\, \|b\|_\M\,.
\]
By definition, $\widetilde\B$ is dense in $\mathcal{C}$ so that $\widetilde \omega$ extends by continuity to $\mathcal{C}$. To prove positivity, recall that, again by Lemma \ref{BdansBtilde}, if $b\ge 0$ we may assume the approximating sequence to be positive so that $\widetilde \omega(b)=\lim \omega(b_n)\geq 0$.
\end{proof}

Next proposition contains approximation and positivity results, we will need in the forthcoming section. They will be also used below to prove that potentials of finite-energy functionals are positive.

\begin{prop}\label{omegaepsilon} Let $\omega\in A^*_+$ be a finite-energy functional, $\widetilde \omega\in \mathcal{C}^*_+$ its canonical extension to the fine algebra $\mathcal{C}$ and $\e>0$. Then
\begin{itemize}
\item i) ${\widetilde \omega\circ (I+\e L)^{-1}}_{|A}$ is a positive finite-energy functional on $A$;
\item ii) one has $G\big({\widetilde \omega\circ (I+\e L)^{-1}}_{|A}\big)=(I+\e L)^{-1}G(\omega)$;
\item iii) one has $(I+L)(I+\e L)^{-1}G(\omega)\in L^1(A,\tau)\cap L^2_+(A,\tau)$;
\end{itemize}
\end{prop}

\begin{proof} As $(I+\e L)^{-1}$ is a positivity preserving, norm contraction on $\M$, the functional $\widetilde \omega\circ (I+\e L)^{-1}$ is positive on $\mathcal C$ and so it is its restriction to $A$, thus proving the statement in i).
\par\noindent
As $(I+\e L)^{-1}(b)\in \D(L)$ for $b\in\B$, the identities
\begin{equation}\label{eqL1}\begin{split}
\widetilde\omega\big( (I+\e L)^{-1}(b)\big)&=\big<G(\omega)\,,\, (I+\e L)^{-1}(b)\big>_\F\\
&=(G(\omega)\,,\, L(I+\e L)^{-1}(b))_{2}+(G(\omega), (I+\e L)^{-1}(b))_{2}\\
&=((I+L)(I+\e L)^{-1}G(\omega)\,,\,b)_{2}\\
&=\big<(I+\e L)^{-1}G(\omega)\,,\, b\big>_\F
\end{split}\end{equation}
allow us to conclude that ${\widetilde \omega\circ (I+\e L)^{-1}}_{|A}$ has finite energy, its potential is given by $G\big({\widetilde\omega\circ (I+\e L)^{-1}}_{|A}\big)=(I+\e L)^{-1}G(\omega)$ and $(I+L)(I+\e L)^{-1}G(\omega)$ is a positive element in $L^2(A,\tau)$.
\par\noindent
The second line in equations (\ref{eqL1}) tells us that the element
\[
h:=(I+L)(I+\e L)^{-1}G(\omega)\in L^2_+(A,\tau)
\]
satisfies
\[
|\tau(hb)|=|(h,b)_2|=|\widetilde \omega((I+\e L)^{-1}b)|\leq \|\widetilde \omega\|_{C^*} \|b\|_A\qquad b\in \B
\]
which suffices to imply $h\in L^1(A,\tau)$ thus proving the first assertion of iii).
\end{proof}

\begin{prop}\label{positivity}
The cone of potentials is contained in the standard cone: $\P\subset L^2_+ (A,\tau)$.
\end{prop}
\begin{proof}
Let us consider a potential $G\in\P$. By the positivity preserving property of the resolvents, we have that $(I+L)^{-1}b\in\F_+ :=\F\cap L^2_* (A,\tau)$ for any  $b\in L^2_* (A,\tau)$ and then
\[
(G,b)_2 = (G,(I+L)(I+L)^{-1}b)_2 = \big<G,(I+L)^{-1}b\big>_\F \ge 0\, .
\]
\end{proof}

Here we prove some useful property shared by potentials.

\begin{lem}\label{potmult}
If $G\in\P$ is a potential then $\displaystyle \frac{1}{\sqrt{G+\delta}}$
is a multiplier of the fine $C^*$-algebra $\mathcal C$, for all $\delta>0$.
\end{lem}

\begin{proof} The function
\[
f:[0,+\infty)\to\mathbb{R}\qquad f(t):= \frac{1}{\sqrt{t+\delta}}-\frac{1}{\delta}
\]
vanishes at $0$, it is bounded and differentiable with bounded derivative. Hence by [\cite{CS1} Lemma 7.2] we have $f(G)\in \widetilde\B\subset \C$. Adding the constant operator $\frac{1}{\delta}$ we get a multiplier of $\mathcal C$.
\end{proof}

\begin{lem} For $\xi,\eta\in \mathcal F$ we have
\begin{equation}\label{der}
\frac{d}{dt}\left<e^{-t(1+L)}\xi,\eta\right>_{L^2(A,\tau)}=-\left<e^{-t(1+L)}\xi,\eta\right>_{\F}\qquad t\ge 0\,.
\end{equation}
\end{lem}

\begin{proof} For $\xi\in Dom_{L^2}(L)$ the identity is obvious. Writing it in integral form
$$\left<e^{-t(1+L)}\xi,\eta\right>_{L^2(A,\tau)}=\left<\xi,\eta\right>_{L^2(A,\tau)} - \int_0^t \left<e^{-s(1+L)}\xi,\eta\right>_{\mathcal F}ds\,,$$
it extends easily to $\xi ,\eta\in\F$.
\end{proof}

\begin{lem}\label{potres} For any potential $G\in \P$ one has
\[
e^{-t(1+ L)}G\leq G \;\text{ in }L^2(A,\tau)\qquad t\ge 0
\]
and
\[
\frac{1}{1+\e L}G \leq \frac{1}{1-\e} G \;\text{ in }L^2(A,\tau)\qquad 0<\e<1\,.
\]
Viceversa, any one of the two above properties implies that $G$ is a potential.
\end{lem}

\begin{proof} Applying (\ref{der}), for $b\in \mathcal F_+$ one has
\begin{equation*}\begin{split}
\frac{d}{dt}\left< e^{-t(1+ L)}G,b\right>_{L^2(A,\tau)} &= -\left< e^{-t(1+ L)}G,b\right>_{\mathcal F}\leq 0
\end{split}\end{equation*}
and then $e^{-t(1+ L)}G\leq G$. Integrating this inequality between $0$ and $+\infty$ with respect to the probability measure $me^{-tm}dt$ for $m>0$, one gets
$$\frac{m}{m+1+L}G \leq G\,,$$
and the result choosing $m$ such that $(m+1)\e=1$. The converse of the two above results are easily obtained deriving the inequalities, weakly in $\F$, in $t=0$ and $\e=0$, respectively.

\end{proof}

We conclude this section with a result that could be considered as a version of a ''noncommutative maximum principle" in Dirichlet spaces (for other versions see [C3], [CS2], [S4]). We will need it in the proof of Proposition \ref{casparticulier} below.

\begin{prop}\label{compare} Let $\omega$ and $\omega'$ in $A^*_+$ be such that $\omega'\leq \omega$ and $\omega$ has finite energy. Then $\omega'$ has finite energy, the potential of $\omega'$ is dominated by the potential of $\omega$
\[
G(\omega')\leq G(\omega)\, ,
\]
meaning that $G(\omega) - G(\omega')\in \F_+$, and the energy content of $\omega'$ is not greater than the one of $\omega$
\[
\E[\omega']\le\E[\omega]\, .
\]
\end{prop}

\begin{proof} If $b\in \B$ is positive one has $\omega'(b)\leq \omega(b)\leq c_\omega ||b||_\F \leq c_\omega ||b||_\F$ for some $c_\omega >0$. Decomposing a generic $b\in \B$ as a linear superposition of positive elements in $\B$ one gets $|\omega(b)|\leq 4c_\omega ||b||_\F$ so that $\omega'$ is a finite-energy functional.
\par\noindent
Notice that, for the same reason, $\omega-\omega'$ is a finite-energy functional on $A$ whose potential is given by  $G(\omega-\omega')=G(\omega)-G(\omega')$. This is a positive element in $L^2(A,\tau)$ by the previous proposition. We conclude the proof by the estimate
\[
\E[\omega']=\omega' (G(\omega'))\le \omega (G(\omega'))\le \omega (G(\omega))=\E[\omega]\, .
\]
\end{proof}

\newpage \section{Deny's embedding and Deny's inequality.}

\medskip
This section is devoted, in present setting of Dirichlet spaces over noncommutative C$^*$-algebras with traces, to prove a theorem obtained by J. Deny ([Den]) in the classical framework.
\par\noindent
What Deny proved is that, if $\mu$ is a finite-energy measure on the locally compact space $X$, having a bounded potential, then the Dirichlet space $\F$ is continuously imbedded in the space $L^2(X,\mu)$. In other words, the Dirichlet form, initially considered as a closed form on $L^2 (X,m)$ with respect to a fixed positive measure $m$, results to be closable on all the spaces $L^2 (\mu ,X)$ with respect to finite-energy measures having bounded potentials. The probabilistic counterpart of this property is the "change of speed measure" or "random time change" of the stochastic Hunt processes $X$ associated to the Dirichlet form and to the different reference measures. A detailed discussion about this can be found in [FOT].
\par\noindent
We will prove below that if $\omega\in A_+^*$ is a finite-energy functional with respect to a Dirichlet form $(\E,\F)$, based on the Hilbert space $L^2 (A,\tau)$ of a trace $\tau$ on $A$, having a bounded potential $G(\omega)\in\M$, then the Dirichlet space $\F$ is embedded in the G.N.S. space $L^2 (A,\omega)$ with embedding norm less than $\sqrt{\|G(\omega)\|_\M}$.
\par\noindent
One of the problem to circumvent in the proof of the result is that, in general, the functional $\omega$ need not to be a trace and consequently the extension of bounded maps on the von Neumann algebra $\M$ to bounded maps on the Hilbert space $L^2 (A,\omega)$ cannot rely on their G.N.S.-symmetry but rather on their K.M.S.-symmetry with respect to $\omega$ (as introduced in \cite{C1}, \cite{C2}). Note that, in general, finite-energy functionals need not to be absolutely continuous with respect to the trace $\tau$ and, as a matter of fact, in current examples most of them are singular with respect to $\tau$.
\vskip0.2truecm\noindent
In the following we will denote by $\Omega\in L^2_+(A,\omega)$ the cyclic vector representing the functional $\omega\in A^*_+$:
\[
\omega (b)=(\Omega ,b\Omega)_{L^2(A,\omega)}\qquad b\in A\, .
\]
We also prove below the Deny's inequality in the noncommutative framework.

\begin{thm}\label{Deny}{\bf (Deny's embedding Theorem)} Let $\omega\in A^*_+$ be a finite-energy functional. If its potential $G(\omega)\in\F$ is bounded, hence belongs to extended Dirichlet algebra $\F\cap\M =\widetilde \B$, then one has
\begin{equation}
\omega(b^*b)\leq ||G(\omega)||_\M \,||b||_\F^2\qquad b\in\B\,.
\end{equation}
Hence, there exist a continuous imbedding $T:\F\to L^2(A,\omega)$, with norm less than $||G(\omega)||_\M^{1/2}$, such that $Tb=b\Omega$ for $b\in \B$.
\end{thm}

\medskip Before proving the theorem in its full generality, we investigate the special case where $\E$ is bounded and $\omega$ is the restriction of a faithful normal functional on $\M$. The general case will be deduced from this special one with help of Proposition 3.6.

\begin{prop}\label{casparticulier} Let $\E$ be a bounded Dirichlet form on $L^2(A,\tau)$ and $\omega\in \M_{*^+}$ be faithful with finite energy. If its potential is bounded $G(\omega)\in \F_+\cap \M$, then one has
\begin{equation}
\omega(b^*b)\leq \|G(\omega)\|_\M \,\|b\|_\F^2\qquad b\in {\widetilde\B}\,.
\end{equation}

\end{prop}

\smallskip
 \begin{proof} The proof proceeds in several steps.

\medskip\noindent
{\it Step 1. Construction of a completely positive kernel.}

\smallskip\noindent
Notice first that, by assumption, there exist $h\in L^1_+(A,\tau)$ such that $\omega(x)=\tau(hx)$ for $x\in \M$. In this case one may realize the G.N.S. representation of $\omega$ in the Hilbert space $L^2 (A,\tau)$ setting $\Omega :=h^{1/2}\in L^2_+(A,\tau)$
\[
\omega(b)=(\Omega, b\,\Omega)_2\qquad b\in \M\, .
\]
One checks easily that $G(\omega)=(I+L)^{-1}h\in L^2 (A,\tau)\cap\M$ and that it is nonsingular: in fact, if $p\in\M$ is the support projection of $G(\omega)$ in $\M$, one has
\[
0=\tau\big(G(\omega)(1_\M-p)\big)=\omega\big((I+L)^{-1}(1_\M-p)\big)\, ,
\]
hence $(I+L)^{-1}(1_\M-p)=0$ by faithfulness of $\omega$ so that $p=1_\M$.
\vskip0.1truecm\noindent
For $x\in \M$, denote $\rho_x\in\M_*$ the $\sigma$-weakly continuous linear form on $\M$ defined by
\[
\rho_x(y)=(Jx^*\Omega,y\Omega)_{2}\qquad y\in\M\,.
\]
By the properties of the standard forms of von Neumann algebras (see \cite{Ara}), if $x\in\M_+$ then
\[
\rho_x(y)=(Jx^*\Omega,y\Omega)_{2}\ge 0\qquad  y\in\M_+
\]
so that $\rho_x\in\M_{*+}$.
The map $\M\ni x\to\rho_x\in\M_*$ is antilinear, $\sigma(\M,\M_*)$-$\sigma(\M_*,\M)$ continuous and satisfies
\[
0\leq \rho_x\leq ||x||\cdot\omega\qquad x\in \M_+\, .
\]
Notice that, since, by assumption, $\E$ is bounded, we have $\F =L^2 (A,\tau)$ and $\widetilde \B = L^2 (A,\tau)\cap\M$. Applying proposition \ref{compare}, we get that $\rho_x$ has finite energy and
\[
G(\rho_x)\leq ||x|| G(\omega)\qquad x\in\M_+\,.
 \]
Since, by assumption, the potential of $\omega$ is bounded, $G(\omega)\in\widetilde\B =L^2(A,\tau)\cap\M$, we have a well defined $\sigma$-weakly continuous, positive linear map $V:\M\to\M$ characterized  by
\[
V(x):=G(\rho_x)\qquad x\in \M_+
\]
and satisfying $V(x)\in \widetilde \B = L^2 (A,\tau)\cap\M$ as well as
\begin{equation}
(Jx^*\Omega,b\Omega)_2=\rho_x(b)=\big<V(x)\,,\,b\big>_\F\qquad x\in\M\, ,\quad b\in \widetilde\B\,.
\end{equation}
We now proceed to check that $V:\M\rightarrow\M$ is a completely positive map. We first check that $V$ is completely positive when considered as a map $V:\M\rightarrow\F$ between the ordered Banach spaces $\M$ and $\F$: for $b_1,\dots,b_n\in\widetilde \B$, $c_1,\dots,c_n\in \M$\,, we compute
\begin{equation*}\begin{split}
\sum_{i,j}\big<V(c_i^*c_j)\,,\,b_i^*b_j\big>_\F &= (Jc_j^*c_i\Omega\,,\,b_i^*b_j\Omega)_2 \\
&=\sum_{i,j}(b_iJc_i\Omega\,,\,b_jJc_j\Omega)_2 \\
&=\|\sum_i\,b_iJc_i\Omega\|_2^2 \geq 0\,.
\end{split}\end{equation*}
This means that, not only $V(x)\in\widetilde\B$ is a potential for any $x\in\M_+$, so that it is positive in $\M$, because of Proposition \ref{positivity}, but also the matrix $[V(c_i^* c_j)]_{i,j=1}^n\in \mathbb{M}_n (\widetilde\B)$ is positive in $\mathbb{M}_n (\M)$ just applying again Proposition \ref{positivity} to the matrix ampliation $\E_n$ of the complete Dirichlet form $\E$ to $L^2 (\mathbb{M}_n(A),\tau_n )$ described, in Definition 2.1 iii).

\smallskip\noindent
Notice that $V(1_\M)=G(\omega)$ so that the endomorphism $V:\M\to\M$ has norm not greater than $\|G(\omega)\|_\M$. More precisely, for $x=x^*\in\M$ one has
$$V(x_+)-V(x)=V(x_-)=G(\rho_{x_-})\geq 0$$
hence $V(x)\leq V(x_+)\leq ||x_+||\,G(\omega)$ and, for sake of symmetry,
\begin{equation}\label{Vborne}
-\|x_-\|\,G(\omega)\leq V(x)\leq \|x_+\|\,G(\omega)\qquad x=x^*\in \M\, .
\end{equation}

\medskip\noindent
{\it Step 2. Reduction of $V$ and $\omega$.}
\par\noindent
Let us consider now the normal, positive functional $\omega':=\rho_{G(\omega)}\in \M_{*+}=L^1_+(A,\tau)$. By the properties of standard forms of von Neumann algebras (see \cite{Ara}), there exists $\Omega'\in L^2_+(A,\tau)$ such that
\begin{equation}
\omega'(x)=(\Omega'\,,\,x\,\Omega' )_{2}\qquad x\in \M\,.
\end{equation}
Moreover, $\|x\Omega'\|^2_2=(x\Omega\,,\,JG(\omega)Jx\Omega)_2\leq \|G(\omega)\|_\M \, \|x\Omega\|^2_2$. Consequently, there exists $\beta'\in \M'$ (the von Neumann algebra commutant of $\M$ in $\B(L^2(A,\tau)$\,) such that $\Omega'=\beta'\,\Omega$ characterized by
\[
\beta' (x\Omega):=x\Omega'\qquad x\in\M\, .
\]

\smallskip\noindent
Notice that, as $\Omega$ and $\Omega'$ belong to the self-polar cone $L^2_+(A,\tau)$ of a standard form, one has $J\Omega=\Omega$ and $J\Omega'=\Omega'$. Setting $\beta=J\beta'J\in J\M'J=\M$, one has $\beta\Omega =\Omega'$.
\par\noindent
Notice also that, as $\omega$ and $\omega'$ are faithful states (by assumption for $\omega$, and by nonsingularity of $G(\omega)$ for $\omega'$) and the vectors $\Omega$ and $\Omega'$ are cyclic and separating, then $\beta$ and $\beta'$ act in $L^2(A,\tau)$ as one to one operators with dense range.
\smallskip\noindent
Then, for $x,y\in \M$ one has
\begin{equation*}\begin{split}
(y\Omega,{\beta'}^*\beta'x\Omega)_2 &= (y\beta'\Omega,x\beta'\Omega)_2 \\
&=(y\Omega'\,,x\Omega')_2 \\ &= \omega'(y^*x) = (JG(\omega) J\Omega,y^*x\Omega)_2 \\
&=(y\Omega\,,\,JG(\omega) J x\Omega)_2
\end{split}\end{equation*}
so that ${\beta'}^*\beta'=JG(\omega)J$ and, finally,
\begin{equation}
\beta^*\beta=G(\omega)\,.
\end{equation}

\smallskip\noindent
As $V$ is completely positive and $V(1_\M)=G(\omega)=\beta^*\beta$, with $\beta$ having initial and final support equal to $1_\M$, there will exist a $\sigma$-weakly continuous completely positive endomorphism $W:\M\to\M$ such that
\begin{equation}
V(x)=\beta^*W(x)\beta\qquad x\in\M\,.
\end{equation}
Moreover, $W(1_\M)=1_\M$, so that $W$ is a noncommutative Markov kernel and, in particular, a contraction of $\M$.

\medskip\noindent {\it Step 3. $\omega'$-KMS-symmetry and $L^2(A ,\omega')$-contractivity of $W$.}
\par\noindent
By the properties of standard forms (see \cite{Ara}), we have, for $x,y\in \M$, the identities
 \begin{equation*}\begin{split}
(  Jy\Omega' \,,\,  W(x)\Omega' )_2&= ( JyJ\beta\Omega  \,,\, W(x)\beta \Omega  )_2\\
&= (Jy\Omega \,,\, \beta^*W(x)\beta\Omega)_2\\
&=(Jy\Omega   \,,\, V(x)\Omega  )_2 \\
&=\big< V(y^*)  \,,\, V(x)  \big>_\F\\
&=\big< V(x^*)  \,,\, V(y)  \big>_\F\\
&=( Jx\Omega'  \,,\,W(y)\Omega')\\
&=( JW(y)\Omega' \,,\,  x\Omega' )_2\,.
\end{split}\end{equation*}
This reveals that $W$ is $\omega'$-KMS-symmetric so that it extends to a bounded map on $L^2(A, \omega')$ by [C2 Proposition 2.24]. As it is a contraction of $\M$, it will be also a contraction in $L^2(A, \omega')$. Alternatively, we can check the boundedness of the extension to $L^2(A, \omega')$ invoking the 2-positivity of $W$:
\begin{equation}\begin{split}
\|W(x)\Omega'\|_2&=(\Omega\,,\,W(x)^*W(x)\Omega')_2 \\
&\leq (\Omega'\,,\,W(x^*x)\Omega') \\
&=(JW(1_\M)\Omega'\,,\,x^*x\Omega')\\
&=(\Omega',x^*x\Omega')_2=\|x\Omega'\|_2^2\qquad x\in\M\,.
\end{split}\end{equation}

\smallskip\noindent Consider now $x\in \M$ and compute
\begin{equation*}\begin{split}
\big<V(x)\,,\,V(x)\big>_\F &= (Jx\Omega\,,\,V(x)\Omega)_{L^2} \\
&=(Jx\Omega'\,,\,W(x)\Omega')_2 \\
&\leq ||x\Omega'||_{L^2(A,\tau)}^2
\end{split}\end{equation*}
so that
\begin{equation}\label{ineqDeny1}
||V(x)||_\F\leq ||x\Omega'||_2\,,\;x\in \M\,.
\end{equation}

\medskip\noindent {\it End of the proof of the proposition.}
For $x$ and $y$ in $\M$, with $y$ such that $\beta^*y\beta\in L^2(A,\tau)$, one computes
\begin{equation*}\begin{split}
\big|\big<Jy\Omega'\,,x\Omega'\big>_2\big| &= \big|\big<J\beta^*y\beta \Omega\,,x\Omega\big>_2\big|\\
&=\big|\big<\beta^*y\beta\,,V(x)\big>_\F \big|\\
&\leq ||\beta^*y\beta||_\F \,||V(x)||_\F \\
&\leq ||\beta^*y\beta||_\F \,||x\Omega'||_2 \text{ by (\ref{ineqDeny1})}\,,
\end{split}\end{equation*}
which provides $||y\Omega'||_2 \leq ||\beta^*y\beta ||_\F$ for all $y\in\M$ and then
\begin{equation}\label{Deny0}
||y\beta\O||_2 \leq ||\beta^*y\beta ||_\F\qquad y\in\M\,.
\end{equation}
As we are assuming that the Dirichlet form is bounded, the $\|\cdot\|_\F$ norm is equivalent to the $L^2 (A,\tau)$ norm. Moreover, since the functional $\o$ is assumed to be faithful, the potential $G(\o)$ has been proved to be nonsingular and, since $\beta^*\beta =G(\o)$, $\beta\in\M$ is nonsingular too. Hence (\ref{Deny0}) extends as
\begin{equation}\label{Deny00}
||x\O||_2 \leq ||\beta^* x ||_\F\qquad x\in \F = L^2 (A,\tau )\,.
\end{equation}
Considering the polar decomposition, there exists a unitary $u\in\M$ such that $\beta^* = G(\o)^{1/2} u^*$ which implies
\begin{equation}\label{Deny000}
||x\O||_2 \leq ||G(\o)^{1/2} u^* x ||_\F\qquad x\in \F = L^2 (A,\tau )
\end{equation}
or
\begin{equation}\label{Deny0000}
||ux\O||_2 \leq ||G(\o)^{1/2} x ||_\F\qquad x\in \F = L^2 (A,\tau )
\end{equation}
and finally
\begin{equation}\label{Deny00000}
||x\O||_2 \leq ||G(\o)^{1/2} x ||_\F\qquad x\in \F = L^2 (A,\tau )
\end{equation}
which provides the result:
$$
\frac{1}{\|G(\o)\|_\M} \|x\O\|_2^2 \le \o (x^* G(\o)^{-1}x)\le \|x\|_\F^2\qquad x\in \F = L^2 (A,\tau )\, .
$$
\end{proof}

%

\medskip\noindent {\it Proof of the theorem.} For $\e>0$, the operator
$$
L_\e = L(I+\e L)^{-1}=\frac{1}{\e}\big(I-(I+\e L)^{-1}\big)
$$
acts as a bounded positive operator in $L^2(A,\tau)$, but also (as it is of the form {\it constant $\times \big($ identity - completely positive contraction $\big)$} it acts on $\M$ as the generator of a semigroup of symmetric completely positive contractions. This means that
\begin{equation}\label{approxDirichletform}
\E_\e\,: L^2 (A,\tau)\to [0,+\infty)\qquad \E_\e[\xi]=\big<\xi, L_\e \xi\big>_2\qquad\xi \in L^2(A,\tau)
\end{equation}
is a bounded symmetric Dirichlet form on $L^2(A,\tau)$.
\par\noindent
The associated Dirichlet space, denoted by $\F_\e$, is the vector space $L^2(A,\tau)$, equipped with the scalar product
$$
\big<\eta\,,\,\xi\big>_{\F_\e}=\big<\eta\,,\,(I+L_\e)\,\xi\big>_2\qquad \xi,\eta\in L^2(A,\tau)\, .
$$
Notice that
\begin{equation}\label{F}
||\xi||_\F=\lim_{\e\downarrow 0} ||\xi||_{\F_\e}\qquad\forall\,\xi\in \F\,.
\end{equation}

\medskip\noindent
Consider now the positive linear form $\widetilde \omega \circ (I+\e L)^{-1}$ on $\mathcal C$, with $\widetilde \omega$ provided by Proposition \ref{omegatilde}. It is well defined since $(I+\e L)^{-1}$ acts as a positive contraction on $L^2(A,\tau)$, hence as a positive contraction of $\F$ (since it commutes with $L$), but also as a $\sigma$-weakly continuous completely positive contraction of $\M$, so that it maps $\widetilde B$ into $\widetilde \B$ and $\mathcal C$ into itself. One has, for $b \in \widetilde \B$,
\begin{equation*}\begin{split}
\widetilde \omega ((I+\e L)^{-1}(b))&=\big<G(\omega)\,,\,(I+\e L)^{-1}b\big>_\F \\
&=\big<(I+\e L)^{-1}G(\omega)\,,\,b\big>_\F \\
&=\tau(h_\e b)
\end{split}\end{equation*}
with $h_\e=(I+L)(I+\e L)^{-1}G(\omega)$ well defined in $L^2(A,\tau)$, since $(I+L)(I+\e L)^{-1}$ is bounded.  One has $\tau(h_\e b)\geq 0$ whenever $b\geq 0$, and $|\tau(h_\e b)|\leq ||\widetilde \omega||_{{\mathcal C}^*}\,||b||_\M$ for any $b\in \widetilde \B$, so that $h_\e\in L^1(A,\tau)_+$ and that $\widetilde \omega \circ (I+\e L)^{-1}$ extends as a normal positive linear form on $\M$.
\par\noindent
The functional $\widetilde \omega \circ (I+\e L)^{-1}$ has finite energy with respect to the Dirichlet form $\E_\e$, and the corresponding potential is
\begin{equation*}\begin{split}
G_\e\big(\widetilde \omega \circ (I+\e L)^{-1}\big)&= (I+L_\e)^{-1}h_\e \\
&= (I+L_\e)^{-1}(I+L)(I+\e L)^{-1}G(\omega) \\
&=\frac{1}{1+\e}G(\omega) + \frac{\e}{1+\e}(1+(1+\e)L)^{-1}G(\omega)
\end{split}\end{equation*}
so that this potential is bounded, with
\begin{equation}\label{Gepsilon}
\big\Vert G_\e\big(\widetilde \omega \circ (I+\e L)^{-1}\big)\big\Vert_\M \leq ||G(\omega)||_\M\, , \; \forall\, \e>0\,.
\end{equation}

\medskip\noindent
As $A$ is separable, there will exist $h_0\in L^2(A,\tau)\cap L^1(A,\tau) \cap \M_+$ which acts as a nonsingular operator on $L^2(A,\tau)$. Let $\omega_0\in\M_{*+}$ be  the corresponding normal positive linear functional on $\M$ defined by $\omega_0(x)=\tau(h_0x)$ for $x\in \M$. Since $\omega_0$ is, by construction, faithful and has finite energy with respect to $\E_\e$, the corresponding potential $G_\e(\omega_0)=(I+L_\e)^{-1}h_0$ is thus bounded, with
\begin{equation}\label{Gepsilon0}
||G_\e(\omega_0)||_\M \leq ||h_0||_\M\,,\;\forall\,\e>0\,.
\end{equation}
Applying now Proposition \ref{casparticulier}  to the Dirichlet form $\E_\e$ and to the faithful, normal, positive linear functional $\widetilde \omega \circ (I+\e L)^{-1}+\e\omega_0\in\M_*$, for all $b\in\B$ we get
\begin{equation}\label{Deny3}\begin{split}
\widetilde \omega \big((I+\e L)^{-1}(b^*b)\big)+\e\omega_0(b^*b) &\leq \big\Vert G_\e( \widetilde \omega \circ (I+\e L)^{-1})+\e G_\e(\omega_0)\big\Vert_\M\, ||b||_{\F_\e}^2\\
&\leq \big(||G(\omega)||_\M+\e||h_0||_\M\big)\,||b||_{\F_\e}^2\,.
\end{split}\end{equation}
As $\e\to 0$, $||b||_{F_\e}^2$ tends to $||b||_\F$\,(cf. (\ref{F})). The convergence in the left hand side is a bit more delicate, since $\omega$ does not necessarily extends as a linear form on $\M$. Nevertheless, for $b\in \B$, we have
\begin{equation*}\begin{split} \lim_{\e \downarrow 0}
\widetilde \omega \big((I+\e L)^{-1}(b^*b)\big)&= \lim_{\e \downarrow 0} \big<G(\omega)\,,\,(I+\e L)^{-1}(b^*b)\big>_\F \\
&= \big<G(\omega)\,,\,b^*b)=\omega(b^*b)
\end{split}\end{equation*}
since $(I+\e L)^{-1}\xi\to \xi$ in $\F$ as $\e \downarrow 0$, for any $\xi\in \F$. Letting $\e\downarrow 0$ in (\ref{Deny3}), we get
$$\omega(b^*b)\leq ||G(\omega)||_\M\,||b||_\F^2\qquad\forall\,b\in \B$$ and the theorem is proved.
\par \hfill $\square$
\begin{rem}
The embedding provided by the above result allows to study the Dirichlet form $\E$ in the space $L^2 (A,\omega)$ of a finite-energy functional having bounded potential. For normal functionals $\omega (a)=\tau (ha)$ this is possible whenever $h\in L^2_+ (A,\tau)\cap \M$ because in that case $G(\omega)=(I+L)^{-1} h\in L^2_+ (A,\tau)\cap \M$. In the case of the Dirichlet integral of a Riemannian manifold $(M,g)$, the associated self-adjoint operator is unitarily equivalent to the Laplace-Beltrami operator of a metric $g'$ which is a conformal change of the original metric $g$. On the noncommutative two torus this point of view has been adopted to study conformal spectral invariants in the setting of Noncommutative Geometry (see \cite{CoTr}).
\end{rem}

The next observation is more important:

\begin{rem}
According to Lemma \ref{potmult}, for $b\in\widetilde \B$, the operator $b^*\frac{1}{G(\o) +\delta}b$ lies in the fine algebra $\C$. Passing to the increasing limit as $\delta\to 0$, one gets $b^*\frac{1}{G(\o) }b$ as a nonnegative operator affiliated to the enveloping von Neumann algebra $\C^{**}$ (cf. \cite{Haa2}).
\par\noindent
Consequently, for all $\o\in\C^*_+$, the quantity $\o (b^*\frac{1}{G(\o) }b )$ is well defined in the extended half line $[0,+\infty]$. In particular, if $\o\in A^*_+$ is a finite-energy functional, it extends as $\widetilde\o$ in $\C^*_+$ and the quantity $\widetilde\o (b^*\frac{1}{G(\o) }b )$ is well defined in the extended half line $[0,+\infty]$. The following Deny's inequality provides a universal bound for this quantity.
\end{rem}
\begin{thm}({\bf Deny's inequality})
For any finite-energy functional $\o\in A^*_+$ the following inequality holds true
\begin{equation}\label{denyeq}
\widetilde\o \Bigl(b^*\frac{1}{G(\o) }b \Bigr)\le \|b\|_\F^2\qquad b\in\widetilde\B\, .
\end{equation}
If the potential is bounded the the inequality is saturated by the choice $b=G(\o)$.
\end{thm}
\begin{proof}
The proof goes through the discussion of several particular cases.
\par\noindent
First particular case: {\it the Dirichlet form $\E$ is bounded, the finite-energy functional $\o\in
A^*_+$ is bounded and its potential $G(\o)\in\P$ is bounded too.} In this case the inequality \ref{denyeq} is just (4.13) or (4.14)
 at the end of the proof of Proposition 4.2.
\par\noindent
Second particular case: {\it the Dirichlet form $\E$ is bounded, the potential $G(\o)\in\P$ of the finite-energy functional $\o\in
A^*_+$is bounded (but $\o$ is not necessarily bounded).} Choose a nonsingular $h_0\in L^1_+ (A,\tau)\cap \M\subset L^2 (A,\tau)$ and consider the functional $\o_0 (\cdot):=\tau (h_0\cdot)$. Then $\o_0$ is faithful, it has finite energy (since $h_0$ lies in $L^2 (A,\tau)$) and it has bounded potential $G(\o_0)=(I+L)^{-1}h_0$ (see Example 3.2). The first particular case applies to $\o +\e\o_0$ so that
$$
(\o +\e\o_0)\Bigl(b^* \frac{1}{G(\o)+\e G(\o_0) +\delta}b\Bigr)\le \|b\|_\F^2\qquad \e\, ,\,  \delta >0\, ,b\in\widetilde\B\, .
$$
Passing to the limit first as $\e\to 0$ and then as $\delta \to 0$ provides the result in this case.
\par\noindent
Third particular case: {\it the Dirichlet form $\E$ is bounded (but neither the finite-energy functional $\o\in A^*_+$ is assumed to be faithful nor its potential $G(\o)\in\P$ is assumed to be bounded).} As $\E$ is bounded, the generator $L$ is a bounded operator on $L^2 (A,\tau)$ so that $\o (\cdot)=\tau (h\cdot)$ where $h\in L^1_+ (A,\tau)\cap L^2 (A, \tau)$ and $h=(I+L)G(\o)$ for $G(\o)\in\P\subset L^2 (A,\tau)$. For any fixed $M>0$, consider $h_M := h\wedge M\in L^1_+ (A,\tau)\cap\M$ and the corresponding finite-energy functional $\o_M (\cdot):=\tau (h_M\cdot)$. One has $G(\o_M)=(I+L)^{-1}h_M \le (I+L)^{-1}h =G(\o)$. According to the second particular case
$$
\o_M \Bigl(b^* \frac{1}{G(\o) + \delta}b\Bigr)\le \o_M \Bigl(b^* \frac{1}{G(\o_M) + \delta}b\Bigr)\le \|b\|_\F^2\qquad \delta >0\, ,b\in\widetilde\B\, .
$$
Passing to the limit first $M\to +\infty$ and then $\delta\searrow 0$ one gets the result in case.
\par\noindent
General case: {\it $\E$ is any Dirichlet form and $\o\in A^*_+$ is any finite-energy functional.} For any $\e>0$, define the functional $\displaystyle \omega_\e=\omega \circ \frac{1}{1+\e L}$ and the bounded Dirichlet form $\mathcal E_\e$ with generator $\displaystyle \frac{L}{1+\e L}$. By Lemma (\ref{potres}), $\omega_\e$ has finite energy with respect to $\mathcal E$, and a fortiori with respect to $\E_\e$.

\smallskip
Let us identify for $b\in \widetilde {\mathcal B}$,
$$
\displaystyle \omega\big(\frac{1}{1+\e L}b)= \left\{\begin{matrix} \displaystyle
=\left<G(\omega),\frac{1}{1+\e L}b\right>_{\mathcal F}=\left<\frac{1}{1+\e L}G(\omega),b\right>_{\mathcal F} \\ \displaystyle
=\left<G_\e(\omega_\e),b\right>_{\mathcal F_\e}=\left<\frac{1}{1+\e L}\frac{1+(1+\e)L}{1+L}G_\e(\omega_e),b\right>_{\mathcal F}
\end{matrix}\right.
$$
so that we get, applying Lemma (\ref{potres}),
\begin{equation*}\begin{split}
G_\e(\omega_\e)&=\frac{1+L}{1+(1+\e)L}G(\omega) \\
&=\frac{1}{1+\e}G(\omega)+\frac{\e}{1+\e}\;\frac{1}{1+(1+\e)L}G(\omega) \\
&\leq{ \frac{1}{1+\e}G(\omega)+\frac{\e}{1+\e}\frac{1}{1-\e}G(\omega)}=\frac{1}{1-\e^2}G(\omega)\,.
\end{split}\end{equation*}
Now the previous particular case allows to write, for any $\delta>0$\,:
\begin{equation*}\begin{split}
(1-\e^2)\omega_\e\big(b^* \,\frac{1}{G(\omega)+\delta}\;b\big)\leq \omega_\e(b^* \frac{1}{G_\e(\omega_\e)}b)\leq ||b||_{\mathcal F_\e}^2\,.
\end{split}\end{equation*}
$$
\omega(b^* \frac{1}{G(\omega)+\delta}\,b)\leq ||b||_{\mathcal F}^2\,.
$$
Passing to the limit first as $\e\to 0$ and then as $\delta\to 0$ provides the result.
\end{proof}

\medskip
As a corollary of the generalized Deny's embedding theorem, we get the following bound which will be used below in Proposition \ref{boundedpotential} and Proposition \ref{critmolt}.
\begin{cor}\label{corDeny} Let us consider a bounded potential $G\in \P\cap \M=\P\cap \widetilde \B$. Then one has
\begin{equation}
\big<G,b^*b\big>_\F \leq ||G||_\M\,||b||_\F^2\,,\;\forall b\in \widetilde \B\,.
\end{equation}
\end{cor}

\smallskip
\begin{proof} When $G=G(\omega)$, with $\omega\in A^*_+$ having finite energy, this is exactly Theorem \ref{Deny}. Now, fix $\e>0$ and consider $G_\e=(I+\e L)^{-1}G$, $h_\e=(I+L)G_\e$. By proposition \ref{omegaepsilon} we have $h_\e\in L^2(A,\tau)_+$\,.

For $\delta>0$, let $p_\delta$ be the spectral projection of $h_\e$ corresponding to the interval $[\delta,+\infty[$. Then, $p_\delta h_\e\in L^1(A,\tau)_+$ and the corresponding linear form $b\to \tau(p_\delta h_\e\,b)$ has a potential $G_{\e,\delta}$ equal to
$$G_{\e,\delta}=(I+L)^{-1}(p_\delta h_\e)\leq (I+L)^{-1}h_\e =G_\e\,.$$
Theorem \ref{Deny} applied to this linear form provides
\begin{equation}
\big<G_{\e,\delta},b^*b\big>_\F \leq ||G||_\M\,||b||_\F^2\,,\qquad\forall b\in \widetilde \B
\end{equation}
since $||G_{\e,\delta}||_\M \leq ||G_\e||_\M\leq ||G||_\M$. The convergence in $\F$, $\lim_{\delta\to 0}G_{\e,\delta}=G_\e$,  is obvious  and we already noticed that $G_\e\to G$ in $\F$ as $\e\to 0$. Hence the result.

\end{proof}

 \section{Energy functionals or  "carr\'{e} du champ" of Dirichlet spaces.}

 \medskip

A Dirichlet forms $(\E ,\F )$ on the space $L^2 (A,\tau)$ of a faithful, semifinite, lower semicontinuous, positive trace $\tau$ on a C$^*$-algebra $A$, gives rise to a family of positive functionals $\{\Gamma [a]\in A^*_+ :a\in\F\}$, called {\it carr\'{e} du champ}, from which the quadratic form can be recovered as
\[
\E [a]=\langle\Gamma [a] ,1_{A^{**}}\rangle\, .
\]
In the noncommutative setting they were introduced in [CS1] to analyze the structure of Dirichlet forms on possibly noncommutative C$^*$-algebras. In the commutative case, where $A=C_0 (X)$, they were defined by Y. Le Jan \cite{LJ} as {\it energy measures}. This appellation being justified by the fact that in applications the positive measure $\Gamma [a]$ may represents the energy distributions over $X$ of the finite-energy configuration $a\in\F$.
\par\noindent
Since in case of the Dirichlet integral on a Riemannian manifold $M$ with measure $m$ one has $\Gamma[a]=|\nabla a|^2\,\cdot m$, they are often called {\it "carr\'{e} du champ"} (even if in general the measure $\Gamma[a]$ is not absolutely continuous with respect to the reference measure of the space $X$).
\vskip0.2truecm\noindent
In this section we show that the carr\'{e} du champ $\Gamma [G]$ of bounded potentials $G\in \P\cap\mathcal{M}$ form a natural class of finite-energy functionals, intimately associated to a Dirichlet space.

%


\medskip\subsection{Energy functionals of a Dirichlet space.}

\begin{defn} (Carr\'e du champ \cite{CS1}).
The {\it carr\'e du champ} $\Gamma[a]\in A^*_+$ of $a\in\B$ is the functional on $A$ defined by
\begin{equation}\label{energyfunctional1}
\big<\Gamma[a],b\big>:=\undemi\bigl(\E(a,ab\star)+\E(ab\star,a)-\E(b\star,a\star a)\bigr)\qquad b\in\B\, .
\end{equation}
It can be shown (see \cite{CS1}) that $\Gamma[a]$ is a bounded positive functional whose norm is $\E[a]$.
\end{defn}

In order to extend the definition to all elements $a\in\F$ of the Dirichlet space and to give a short proof of the main result of this section,
we briefly recall the main properties of the differential calculus associated to a regular Dirichlet form (see \cite{CS1}, \cite{C2}),
in terms of which an alternative and more manageable form of $\Gamma[a]$ can be given.
\par\noindent
Any regular Dirichlet form $(\E ,\F)$ on $L^2 (A,\tau)$ can be described as
\[
\E[a]=\|\partial a\|^2_\H\qquad a\in \F
\]
by a map $\partial :\F\to\H$ which is closed on $L^2 (A,\tau)$, takes its values in a Hilbert $A$-$A$-bimodule $\H$ and which is a derivation on the Dirichlet algebra $\B\subseteq\F$, in the sense that satisfies the Liebniz rule
\[
\partial (ab)=(\partial a)\cdot b + a\cdot (\partial b)\qquad a,b\in\B
\]
(the dots denote the left and right actions of elements of $\B$ on vectors in $\H$). Moreover, on the bimodule there exists a symmetry $\J :\H\to\H$, i.e. an antiunitary involution which intertwines the left and right actions of $A$
\[
\J (a\xi b)=b^* (\J\xi)a^*\qquad a,b\in A\, ,\quad \xi\in\H\, ,
\]
such that
\[
\partial (a^*)=\J (\partial a)\qquad a\in A\, .
\]
Summarizing, one describes the self-adjoint, nonnegative operator $L$ on $L^2 (A,\tau)$ whose quadratic form is the Dirichlet form $(\E,\F)$ as the {\it divergence of a derivation}: $L=\partial^*\circ\partial$ or, in other words, one can refers to the derivation as the {\it differential square root of the generator $L$}. The derivation representing a regular Dirichlet form is essentially unique (see \cite{CS1} Theorem 8.3 for details).

\begin{ex}{\bf Derivation associated to negative definite functions on group C$^*$-algebras.}
In Example 2.5 we considered the Dirichlet form $\E_\ell$ on the reduced group C$^*$-algebra C$^*_{red} (G)$ of a locally compact group $G$, associated to a continuous negative definite functions $\ell :G\to [0,+\infty)$. To describe the derivation it gives rise, recall that there exists a 1-cocyle $(\pi ,\K ,c)$, where $\pi :G\to\K$ is an orthogonal representation of $G$ in some real Hilbert space $\K$ and $c:G\to\K$ is a continuous function satisfying
\[
c(st)=c(s)+\pi (s)c(t)\qquad s,t\in G\, ,
\]
such that $\ell (s)=\|c(s)\|_\K^2$ for all $s\in G$. Denote by $\K_\mathbb{C}$ the complexification of the real Hilbert space $\K$ and by $\K_\mathbb{C}\ni\xi\mapsto \overline{\xi}\in\K_\mathbb{C}$ its canonical conjugation. The tensor product of complex Hilbert spaces $\K_\mathbb{C}\otimes L^2 (G)$ is a C$^*_{red} (G)$-bimodule under the commuting actions $\pi_l :=\pi\otimes\lambda$ and $\pi_r :=id\otimes \rho$ constructed by the left and right regular representations $\lambda\, ,\rho$ of  C$^*_{red} (G)$ in $L^2 (G)$. This bimodule structure turns out to be symmetric with respect to the anti-linear involution given by
\[
\J (\xi\otimes a):= \overline{\xi}\otimes J(a)\qquad \xi\otimes a\in \K_\mathbb{C}\otimes L^2(G)\, ,
\]
where $J(a)(s)=\overline{a(s^{-1})}$, $s\in\Gamma$, is just the involution associated to the standard cone of positive definite functions in $L^2 (G)$.
As customary, the same symbol $\pi$ will denote both the unitary representation of $\Gamma$ and the induced representation of $C\star_{red}(\Gamma)$.
The map $\partial :D(\partial )\to \K_\mathbb{C}\otimes L^2 (G)$ defined by
\[
\qquad D(\partial):=C_c (G)\, ,\qquad \partial (a):=c\otimes f\, ,\qquad a\in C_c (G)\, ,
\]
is the a closable derivation such that
\[
\E [a]=\| \partial a\|^2_{\K_\mathbb{C}\otimes L^2 (G)}\qquad a\in D(\partial)\subseteq \F_\ell\, .
\]
See \cite{CS1}, \cite{C2} for the details.
\end{ex}

\begin{ex}{\bf Derivation on noncommutative tori.}
The derivation associated to the Dirichlet form we introduced in Section 2 Example 2.6 and given by
\[
\E\Bigl[\sum_{n,m\in\mathbb{Z}}\alpha_{n,m}U^nV^m\Bigr] =
\sum_{n,m\in\mathbb{Z}}(n^2+m^2)|\alpha_{n,m}|^2
\]
on the noncommutative torus $A_\theta$ is the direct sum
\[
\partial(a)=\partial_1(a)\oplus \partial_2(a)
\]
of the following derivations $\partial_1$ and $\partial_2$ defined by
\[
\partial_1(U^nV^m)=inU^nV^m \; , \qquad
\partial_1(U^nV^m)=imU^nV^m\qquad n,m\in
\mathbb{Z}\, .
\]
The $A_\theta$--bimodule $\H$ associated with $\E$ is a sub-bimodule of the direct sum $L^2(A,\tau)\oplus L^2(A,\tau)$ of two copies of the standard $A_\theta$-bimodule.
\end{ex}

\vskip0.2truecm\noindent
The following lemma contains consequences of the crucial observation that a Dirichlet form which is regular with respect to the C$^*$-algebra $A$ is also automatically regular with respect to the fine C$^*$algebra $\mathcal{C}$.
\begin{lem}\label{finecalculus}
Let $(\E ,\F)$ be a Dirichlet form on $L^2 (A,\tau)$ which is regular with respect to the C$^*$-algebra $A$.
\par\noindent
Then the trace $\tau$ on $A$ naturally extends to a trace on the fine C$^*$-algebra $\mathcal{C}$ so that the G.N.S. representation of $(\mathcal{C},\tau)$ is an extension of the G.N.S. representation of $(A,\tau)$ and, in particular, $L^2 (\mathcal{C},\tau) = L^2 (A,\tau) =L^2 (\M,\tau)$.
\par\noindent
Moreover, since $\C\cap\F\supseteq \widetilde{\B}\cap\F = \widetilde{\B}$, the Dirichlet form $(\E,\F)$ is also regular with respect to the C$^*$-algebra $\mathcal{C}$.
\par\noindent
As a consequence, the differential calculus $(\widetilde{\partial},\widetilde{\B},\widetilde{\H},\widetilde{\J})$, associated to $(\E,\F)$ on $(\C,\tau)$ is an extension of the corresponding one $(\partial , \B ,\H ,\J)$ on $(A,\tau)$. In particular, once these calculi have been identified, the Leibniz rule holds true on the extended Dirichlet algebra $\widetilde{\B}$
\[
\partial (ab)=(\partial a)\cdot b + a\cdot (\partial b)\qquad a,b\in\widetilde{\B}\, .
\]
\end{lem}
\begin{proof}
Notice that, even if the fine C$^*$-algebra $\C$ need not to be separable, it acts, by definition, on a separable Hilbert space so that it admits a faithful state and the framework of \cite{CS1} applies.\par\noindent
The first statement concerning the trace  comes from the fact that, by definition, $A\subseteq \C\subseteq \M$ so that the normal extension of the trace $\tau$ to the von Neumann algebra $\M$ reduce to a trace on the subalgebra $\C$. The second one follows because, by definition, the Dirichlet algebra $\C\cap\F$ of $(\E ,\F)$ with respect to $(\C,\tau)$ contain the extended Dirichlet algebra $\widetilde{\B}$ and this one is, again by definition, dense in $\C$.
\end{proof}
As announced before, using the derivation associated to a Dirichlet space, one can readily give a definition of the energy functional $\Gamma[a]$ for all elements $a\in\F$ by
\begin{equation}\label{energyfunctional2}
\big<\Gamma[a], b\big>_{\C^*,\C}=\big<\partial a,(\partial a)\cdot b\big>_\H\qquad b\in \C \,.
\end{equation}
Using the Leibniz rule one can check that the above formula reduce to (\ref{energyfunctional1}) whenever $a,b\in\B$.

The following result shows that the family of finite-energy functionals include some natural functional deeply connected to the structure of the Dirichlet space.

\begin{prop}\label{boundedpotential} If $G$ is a bounded potential, $G\in \P\cap \M=\P\cap \widetilde \B$, its carr\'{e} du champ $\Gamma[G]\in\C^*_+$ is a finite-energy functional.
\end{prop}

\begin{proof}
Let us consider on the extended Dirichlet algebra, the functional $\o_G :\widetilde\B\to\mathbb{C}$ defined by the potential $G\in\P\cap \widetilde \B$:
\[
\o_G :\widetilde\B\to\mathbb{C}\qquad \o_G (b):=\big<G,b\big>_\F\, .
\]
Since the Dirichlet form is completely positive, the functional $\o_G$ is completely positive with respect to the cone $\P\subset\widetilde\B$. Therefore a Cauchy-Schwartz inequality holds true
\begin{equation}
|\o_G (b^*c)|^2\le \o_G (b^*b)\cdot\o_G (c^*c)\qquad b,c\in\widetilde\B\, .
\end{equation}
Hence we have
\begin{equation}
|\big< G,Gb\big>_\F|^2\le \big< G,G^2\big>_\F\cdot \big< G,b^*b\big>_\F\qquad b\in\widetilde\B
\end{equation}
and by Corollary \ref{corDeny} we have also
\begin{equation}
|\big< G,Gb\big>_\F|\le\|G\|_\M\cdot \|G\|_\F\cdot \|b\|_\F\qquad b\in\widetilde\B\, .
\end{equation}
Then we compute for $b\in\widetilde\B_+$
\begin{equation*}
\begin{split}
\Gamma[G](b)&=\big<\partial(G),\partial(G)b\big>_\H\\
&=\big<\partial(G),\partial(Gb)\big>_\H-\big<\partial(G),G\partial(b)\big>_\H\\
&=\E (G,Gb)-\big<G\partial(G),\partial(b)\big>_\H \\
&\le \langle G, Gb\rangle_\F +\|G\|_\M\sqrt{\E[G]}\cdot\sqrt{\E[b]}\qquad b\in\widetilde\B\\
&\le \langle G, Gb\rangle_\F +\|G\|_\M \|G\|_\F\cdot\|b\|_\F\\
&\le \|G\|_\M\cdot \|G\|_\F\cdot \|b\|_\F +\|G\|_\M \|G\|_\F\cdot\|b\|_\F\\
&= 2\|G\|_\M\cdot \|G\|_\F\cdot \|b\|_\F\\
\end{split}
\end{equation*}
which provides the result.

\end{proof}


\section{Multipliers of Dirichlet spaces}

We define in this section {\it multipliers of Dirichlet spaces} and, as a final application of the previous work, we prove their existence and a related approximation property.

\begin{defn}(Multipliers of a Dirichlet space)
An element $b\in\mathcal{M}$ is called a {\it multiplier} of the Dirichlet space $(\E ,\F)$ if
\[
b\xi\in\F \quad\text{ and }\quad\xi b\in\F\qquad \forall \,\, \xi\in \F\,.
\]
A direct application of the Closed-Graph Theorem implies that multipliers are bounded maps on the Dirichlet space $\F$ and form an involutive sub-algebra, denoted by $\M (\E,\F)$, of the algebra $\mathbb{B}(\F)$ of all bounded operators on $\F$.
\par\noindent
Notice that if the Dirichlet space contains the unit $1_\M\in\F$, then the multipliers algebra is a subalgebra of the extended Dirichlet algebra: $\M (\E,\F)\subseteq \widetilde{\B}$.
\end{defn}
We prove below that multipliers exist.

\medskip
\begin{prop}\label{critmolt}
Let $g\in\P\cap\M$ be a bounded potential and suppose that its carr\'{e} du champ $\Gamma [g]\in\C^*_+$ has a bounded potential $G(\Gamma[g])\in\P\cap\M$. Then $g$ is a multiplier of the Dirichlet space.
\end{prop}

\begin{proof} Applying the generalized Deny embedding Theorem \ref{Deny} and Proposition \ref{boundedpotential}, we get, for $b\in\B$:
\begin{equation}\label{ineqmolt}\begin{split}
||(\partial g)b||_\H^2 &=\big<\Gamma [b],bb^*\big>_{\mathcal C^*,\mathcal C}\\
& \leq ||G(\Gamma [b])||_\M\,||b^*||_\F^2=||G(\Gamma [g]||_\M\,||b||_\F^2\,.
\end{split}\end{equation}
Hence
\[
||\partial (gb)||_\H=||\partial(g)b+g\partial(b)||_\H
\leq \big(||G(\Gamma[g])||_\M^{1/2}+ ||g||_\M\big)\,||b||_\F
\]
and then
\[
||gb||_\F^2 = ||\partial (gb)||_\H^2 + ||gb||_{L^2 (A,\tau)}^2
\leq \bigl[\bigl(||G(\Gamma[g])||_\M^{1/2}+ ||g||_\M\bigr)^2 + ||g||_\M^2\bigl]\,||b||_\F^2\, .
\]
Since the Dirichlet algebra $\B$ is a form core, for a fixed $b\in\F$ the exists a Cauchy net $\{b_i\in\B :i\in I\}$ converging to it in the norm of $\F$. The above bound implies that also $\{gb_i\in\B :i\in I\}\subset\F$ is a Cauchy net in $\F$, hence converging to an element $c\in\F$. Since $\F$ is continuously embedded in $L^2 (A,\tau)$, we have that $c=gb$. An analogous computation shows that $bg\in\F$ for all $b\in\F$ so hat $g$ is a multiplier of the Dirichlet space.
\end{proof}

Next result shows that the resolvent $(I+L)^{-1}$ are positivity preserving maps from the Hilbert algebra $L^2(A,\tau)\cap \M$ into the multipliers algebra $\M (\E,\F)$.

\begin{prop} Let $h\in L^2(A,\tau)\cap \M$. Then $g=(I+L)^{-1}h\in\M$ is a potential and a multiplier of the Dirichlet space $\F$.
\end{prop}

\begin{proof} Without loss of generality, we may assume $h\in L^2 (A,\tau)_+\cap \M$. Since $\big<g,b\big>_\F=\tau((I+L)g\cdot b)=\tau(hb)$ for all $b\in \widetilde \B$, we have that $g$ is a bounded potential (see Example 3.2).
\par\noindent
Denoting by $\J$ the anti-unitary involution of $L^2 (A,\tau)$ determined by the self-polar cone $L^2 (A,\tau)$, since $g=g^*$, for all $b\in \widetilde \B$ we have
\[
\big<\partial g,(\partial g)b\big>_\H  = \big<\J ((\partial g)b),\J (\partial g)\big>_\H = \big<b^*(\partial g),\partial g\big>_\H  = \big<\partial g,b(\partial g)\big>_\H
\]
and then
\begin{equation*}\begin{split}
2\big<\Gamma[g],b\big>_{\mathcal C^*,\mathcal C}& =2\big<\partial g,(\partial g)b\big>_\H \\
&=2\big<\partial g,b(\partial g)\big>_\H \\
&=\big<\partial g,(\partial g)b+b(\partial g)\big>_\H \\
&=\big<\partial g,\partial( gb)+\partial(b g)-g(\partial b)-(\partial b)g\big>_\H \\
&=\tau\big(h(gb+b g)\big)-\big<g(\partial g)+(\partial g)g,\partial b\big>_\H \\
&=\tau\big((hg+gh)b \big)-\big<\partial g^2\,,\,\partial b\big>_\F\\
&=\big<(I+L)^{-1}(hg+gh)-g^2,b\big>_\F
\end{split}\end{equation*}
which provides that the positive linear functional $\Gamma[g]$ has a bounded potential $(I+L)^{-1}(hg+gh)-g^2\in\M$. Apply Proposition \ref{critmolt} to conclude.
\end{proof}

\begin{cor} Let $g$ be a bounded potential. 1. Then, for any $\e>0$, $(I+\e L)^{-1}g$ is a multiplier of the Dirichlet space $\F$. 2. Multipliers are dense in $\F$. 3. The algebra of multipliers is dense in fine C$^*$-algebra $\C$.

\end{cor}
\begin{proof} Apply the previous corollary and Lemma 2.3.
\end{proof}


\newpage
\normalsize
\begin{center} \bf REFERENCES\end{center}

\normalsize
\begin{enumerate}

\bibitem[AHK]{AHK} S. Albeverio, R. Hoegh-Krohn, \newblock{Dirichlet
Forms and Markovian semigroups on C$^*$--algebras},
\newblock{\it Comm. Math. Phys.} {\bf 56} {\rm (1977)}, 173-187.

\bibitem[Ara]{Ara} H. Araki, \newblock{Some properties of modular conjugation operator of von Neumann algebras and a non-commutative Radon-Nikodym theorem with a chain rule},
\newblock{\it Pacific J. Math.} {\bf 50} {\rm (1974)}, 309-354.

\bibitem[BeDe1]{BeDe1} A. Beurling and J. Deny, \newblock{Espaces de Dirichlet
I: le cas \'el\'ementaire},
\newblock{\it Acta Math.} {\bf 99} {\rm (1958)}, 203-224.

\bibitem[BeDe2]{BeDe2} A. Beurling and J. Deny, \newblock{Dirichlet spaces},
\newblock{\it Proc. Nat. Acad. Sci.} {\bf 45} {\rm (1959)}, 208-215.

\bibitem[Bi]{Bi} P. Biane, \newblock{Logarithmic Sobolev Inequalities, Matrix Models and Free Entropy},
\newblock{\it Acta Math. Sinica, English Series.} {\bf 19} {\rm (2003)}, 497-506.

\bibitem[Boz]{Boz} M. Bozejko, \newblock{Positive definite functions on the
free group and the noncommutative Riesz product},
\newblock{\it Bollettino U.M.I.} {\bf 5-A} {\rm (1986)}, 13-21.

\bibitem[Bre]{Bre}M. Brelot, \newblock{La theorie moderne du potentiel},
\newblock{\it Ann. Inst. Fourier Grenoble} {\bf 4} {\rm (1952)}, 113-140.

\bibitem[Ca]{Ca} H. Cartan, \newblock{Sur les fondements de la théorie du potentiel},
\newblock{\it Bull. Soc. Math. France} {\bf 69} {\rm (1941)}, 71--96.

\bibitem[CCJJV]{CCJJV} P.-A. Cherix, M. Cowling, P. Jolissaint, P. Julg, A. Valette,
\newblock{``Groups with the Haagerup property. Gromov's a-T-menability''},
\newblock{Progress in Mathematics, 197, Birkhäuser Verlag, Basel, 2001}

\bibitem[C1]{C1} F. Cipriani, \newblock{Dirichlet forms and Markovian semigroups on standard forms of von Neumann algebras},
\newblock{\it J. Funct. Anal.} {\bf 147} {\rm (1997)}, no. 1, 259--300.

\bibitem[C2]{C2} F. Cipriani, \newblock{``Dirichlet forms on Noncommutative spaces''},
\newblock{Springer ed. L.N.M. 1954, 2007}.

\bibitem[C3]{C3} F. Cipriani, \newblock{The variational approach to the
Dirichlet problem in C$^*$--algebras},
\newblock{\it Banach Center Publications} {\bf 43} {\rm (1998)}, 259-300.

\bibitem[CFK]{CFK} F. Cipriani, U. Franz, A. Kula, \newblock{Symmetry properties of quantum MArkov semigroups on Compact Quantum Groups},
\newblock{\it In preparation}.

\bibitem[CS1]{CS1} F. Cipriani, J.-L. Sauvageot, \newblock{Derivations as square
roots of Dirichlet forms},
\newblock{\it J. Funct. Anal.} {\bf 201} {\rm (2003)}, no. 1, 78--120.

\bibitem[CS2]{CS2} F. Cipriani, J.-L. Sauvageot,  \newblock{Strong solutions to the Dirichlet problem for differential forms: a quantum dynamical semigroup approach},
\newblock{\it Contemp. Math, Amer. Math. Soc., Providence, RI} {\bf 335} {\rm (2003)}, 109-117.

\bibitem[CS3]{CS3} F. Cipriani, J.-L. Sauvageot, \newblock{Noncommutative potential theory and
the sign of the curvature operator in Riemannian geometry},
\newblock{\it Geom. Funct. Anal.} {\bf 13} {\rm (2003)}, no. 3, 521--545.

\bibitem[CGIS1]{CGIS1} F. Cipriani, D. Guido, T. Isola, J.-L. Sauvageot,  \newblock{``Differential 1-forms, their Integrals and Potential Theory on the Sierpinski Gasket ''},
\newblock{arXiv:1105.1995, 2011}.

\bibitem[CGIS2]{CGIS2} F. Cipriani, D. Guido, T. Isola, J.-L. Sauvageot,  \newblock{``Spectral triples for the Sierpinski Gasket''},
\newblock{arXiv:1112.6401, 2011}.

\bibitem[Co]{Co} A. Connes, \newblock{``Noncommutative Geometry''},
\newblock{Academic Press, 1994}.

\bibitem[CoTr]{CoTr} A. Connes, P. Tretkoff, \newblock{The Gauss-Bonnet Theorem
for the noncommutative two torus},
\newblock{Noncommutative geometry, arithmetic, and related topics, Johns Hopkins Univ. Press, Baltimore, MD,, 2011}.

\bibitem[Da]{Da} Y. Dabrowski, \newblock{A note about proving non-Ã under a finite
non-microstates free Fisher information assumption},
\newblock{\it Math. Z.} {\bf 258} {\rm (2010)}, 3662-3674.

\bibitem[DL]{DL} E.B. Davies, J.M. Lindsay, \newblock{Non--commutative
symmetric Markov semigroups},
\newblock{\it Math. Z.} {\bf 210} {\rm (1992)}, 379-411.

\bibitem[DR1]{DR1} E.B. Davies, O.S. Rothaus, \newblock{Markov semigroups on C$^*$--bundles},
\newblock{\it J. Funct. Anal.} {\bf 85} {\rm (1989)}, 264-286.

\bibitem[DR2]{DR2} E.B. Davies, O.S. Rothaus, \newblock{A BLW inequality for vector bundles and applications to spectral bounds},
\newblock{\it J. Funct. Anal.} {\bf 86} {\rm (1989)}, 390-410.

\bibitem[deH]{deH} P. de la Harpe, \newblock{``Topics in Geometric Group Theory''},
\newblock{Chicago Lectures in Mathematics, The University of Chicago Press, 2000}.

\bibitem[Den]{Den} J. Deny, \newblock{M\'ethodes hilbertien en théorie du potentiel},
\newblock{\it Potential Theory (C.I.M.E., I Ciclo, Stresa)}, \newblock{Ed. Cremonese Roma, 1970}. {\bf 85}, 121-201.

\bibitem[Dix]{Dix} J. Dixmier, \newblock{``Les C$^*$--alg\`ebres et leurs
repr\'esentations''}, \newblock{Gauthier--Villars, Paris, 1969}.

\bibitem[Do]{Do} J.L. Doob, \newblock{``Classical potential theory and its probabilistic counterpart''},
\newblock{Springer-Verlag, New York, 1984}.

\bibitem[F1]{F1} M. Fukushima, \newblock{Regular representations of Dirichlet spaces},
\newblock{\it Trans. Amer. Math. Soc.} {\bf 155} {\rm (1971)}, 455-473.

\bibitem[F2]{F2} M. Fukushima, \newblock{Dirichlet spaces and strong Markov processes},
\newblock{\it Trans. Amer. Math. Soc.} {\bf 162} {\rm (1971)}, 185-224.

\bibitem[FOT]{FOT} M. Fukushima, Y. Oshima, M. Takeda,
\newblock{``Dirichlet Forms and Symmetric Markov Processes''},
\newblock{de Gruyter Studies in Mathematics, 1994}.

\bibitem[G1]{G1} L. Gross, \newblock{Existence and uniqueness of physical
ground states},
\newblock{\it J. Funct. Anal.} {\bf 10} {\rm (1972)}, 59-109.

\bibitem[G2]{G2} L. Gross, \newblock{Hypercontractivity and logarithmic
Sobolev inequalities for the Clifford--Dirichlet form},
\newblock{\it Duke Math. J.} {\bf 42} {\rm (1975)}, 383-396.

\bibitem[Haa1]{Haa1} U. Haagerup, \newblock{An example of a nonnuclear C$^*$-algebra, which has the metric approximation property},
\newblock{\it Invent. Math.} {\bf 50} {\rm (1978)}, no. 3, 279-293.

\bibitem[Haa2]{Haa2} U. Haagerup, \newblock{Operator-valued weights in von Neumann algebras. I},
\newblock{\it J. Funct. Anal.} {\bf 32} {\rm (1979)}, no. 2, 175-206.

\bibitem[LJ]{LJ} Y. Le Jan, \newblock{Mesures associ\'es a une forme de Dirichlet. Applications.},
\newblock{\it Bull. Soc. Math. France} {\bf 106} {\rm (1978)}, 61-112.

\bibitem[Moko]{Moko} G. Mokobodzki, \newblock{Fermabilit\'e des formes de Dirichlet et in\'egalit\'e de type Poincar\'e},
\newblock{\it Pot. Anal.} {\bf 4} {\rm (1995)}, 409-413.


\bibitem[Pe1]{Pe1} J. Peterson, \newblock{$L^2$-rigidity in von Neumann algebras},
\newblock{\it Invent. Math.} {\bf 175} {\rm (2009)}, 417-433.

\bibitem[Pe2]{Pe2} J. Peterson, \newblock{A 1-cohomology characterization of property (T) in von Neumann algebras},
\newblock{\it Pacific J. Math.} {\bf 243} {\rm (2009)}, no. 1, 181-199.

\bibitem[S1]{S1} J.-L. Sauvageot, \newblock{Tangent bimodule and locality
for dissipative operators on C$^*$--algebras, Quantum Probability
and Applications IV}, \newblock{\it Lecture Notes in Math.} {\bf 1396} {\rm (1989)}, 322-338.

\bibitem[S2]{S2} J.-L. Sauvageot, \newblock{Quantum Dirichlet forms,
differential calculus and semigroups,  Quantum Probability and Applications V},
\newblock{\it Lecture Notes in Math.} {\bf 1442} {\rm (1990)}, 334-346.

\bibitem[S3]{S3} J.-L. Sauvageot, \newblock{Semi--groupe de la chaleur
transverse sur la C$^*$--alg\`ebre d'un feulleitage riemannien},
\newblock{\it C.R. Acad. Sci. Paris S\'er. I Math.} {\bf 310} {\rm (1990)}, 531-536.

\bibitem[S4]{S4} J.-L. Sauvageot, \newblock{Semi--groupe de la chaleur transverse sur la C$^*$--alg\`ebre d'un feulleitage riemannien},
\newblock{\it J. Funct. Anal.} {\bf 142} {\rm (1996)}, 511-538.

\bibitem[V1]{V1} D.V. Voiculescu, \newblock{Lectures on Free Probability theory.}, \newblock{\it Lecture Notes in Math.} {\bf 1738} {\rm (2000)}, 279-349.

\bibitem[V2]{V2} D.V. Voiculescu, \newblock{The analogues of entropy and of Fisher's information measure in free probability theory},
\newblock{\it Invent. Math.} {\bf 132} {\rm (1998)}, 189-227.

\bibitem[V3]{V3} D.V. Voiculescu, \newblock{Almost Normal Operators mod Hilbert–Schmidt and the K-theory of the Algebras $E\Lambda (\Omega)$},
\newblock{\it arXiv:1112.4930}.

\end{enumerate}
\end{document}